\documentclass{article}

\usepackage{graphicx}

\usepackage{lipsum}
\usepackage{amsfonts}
\usepackage{bbm}
\usepackage{graphicx}
\usepackage{caption}
\usepackage{epstopdf}
\ifpdf
  \DeclareGraphicsExtensions{.eps,.pdf,.png,.jpg}
\else
  \DeclareGraphicsExtensions{.eps}
\fi

\usepackage{lscape} %

\usepackage{enumitem}
\setlist[enumerate]{leftmargin=5mm}

\usepackage{amsopn}

\usepackage[linesnumbered,algoruled,vlined,algo2e]{algorithm2e}
\SetAlCapNameFnt{\footnotesize}
\SetAlCapFnt{\footnotesize}

\usepackage{amssymb,amsfonts,amscd,mathtools}

\DeclarePairedDelimiterX\Set[2]{\lbrace}{\rbrace}%
 { #1 \,\delimsize| \,\mathopen{} #2 }

\newcommand{\ck}{\kappa}
\newcommand{\ckeg}{\kappa_{\mathrm{eg}}}
\newcommand{\ckef}{\kappa_{\mathrm{ef}}}
\newcommand{\derate}{\gamma_{\rm dec}}
\newcommand{\inrate}{\gamma_{\rm inc}}
\newcommand{\up}[1]{^{(#1)}}
\newcommand{\ol}{\overline}

\newcommand{\I}{\mathcal{I}}

\newcommand{\wh}{\widehat}
\newcommand{\mc}{\mathcal}
\newcommand{\mb}{\mathbb}
\newcommand{\mR}{\mb{R}}
\newcommand{\bs}[1]{\boldsymbol{#1}}

\newcommand{\E}{\mathbb{E}}

\renewcommand{\P}{\mathbb{P}}
\newcommand{\R}{\mathbb{R}}
\newcommand{\Qx}{Q^{1-\alpha}(x)}
\newcommand{\Qv}[1]{Q^{1-\alpha}(#1)}
\newcommand{\Q}[2]{\wh{Q}^{1-\alpha}_{#1}(#2)}

\newcommand{\norm}[1]{\left\lVert#1\right\rVert}
\newcommand{\abs}[1]{\left\lvert#1\right\rvert}
\newcommand{\parenth}[1]{\left(#1\right)}

\newcommand{\parenthcurl}[1]{\left\{#1\right\}}

\newcommand{\dom}{\mathbb{S}}
\newcommand{\Cdom}{C_{\dom}}
\newcommand{\Kdom}{K_{\dom}}
\newcommand{\gN}{g_{0,N}}
\newcommand{\etaprod}{\eta_1\eta_2}

\newcommand{\inner}[2]{\left\langle #1,#2 \right\rangle}
\newcommand{\minimize}{\operatornamewithlimits{minimize}}
\newcommand{\maximize}{\operatornamewithlimits{maximize}}
\newcommand{\defined}{=}

\newcommand{\mumax}{\mu_{\mathrm{max}}}
\newcommand{\bigcurl}[1]{\left\{#1\right\}}

\usepackage{fullpage,amsthm,xcolor}
\usepackage[breaklinks=true, colorlinks=true, linkcolor={blue!90!black}, citecolor={blue!90!black}, urlcolor={blue!90!black}]{hyperref}
\newtheorem{definition}{Definition}[section]
\newtheorem{theorem}{Theorem}[section]
\newtheorem{lemma}{Lemma}[section]
\newtheorem{corollary}{Corollary}[section]
\newtheorem{remark}{Remark}[section]
\newtheorem{example}{Example}[section]
\newtheorem{assumption}{Assumption}[section]
\newtheorem{proposition}{Proposition}[section]
\usepackage{algorithm}
\usepackage[nameinlink]{cleveref}
\usepackage{xcolor,url,doi}
\newcommand{\subassref}[2]{{\rm\Cref{ass:#1}(\ref{subass:#2})}}
\creflabelformat{assumption}{#2#1#3}
\creflabelformat{definition}{#2#1#3}
\creflabelformat{algocf}{#2#1#3}
\creflabelformat{lemma}{#2#1#3}
\creflabelformat{proposition}{#2#1#3}

\newcommand{\bdm}{\begin{displaymath}}
\newcommand{\edm}{\end{displaymath}}
\newcommand{\beq}{\begin{equation}}
\newcommand{\eeq}{\end{equation}}
\newcommand{\ba}{\begin{aligned}}
\newcommand{\ea}{\end{aligned}}
\DeclarePairedDelimiter{\ceil}{\lceil}{\rceil}

\usepackage{thmtools}
\usepackage{nameref}
\crefname{lemma}{lemma}{lemmas}
\Crefname{lemma}{Lemma}{Lemmas}
\crefname{assumption}{assumption}{assumptions}
\Crefname{assumption}{Assumption}{Assumptions}
\crefname{example}{Ex.}{Ex.}
\crefname{algocf}{alg.}{algs.}
\Crefname{algocf}{Algorithm}{Algorithms}
\Crefname{algocfline}{Line}{Lines}

\usepackage[normalem]{ulem} %

\usepackage{cite}

\newcommand{\email}{\url}
\begin{document}

\title{An Empirical Quantile Estimation Approach to Nonlinear Optimization Problems with Chance Constraints}

\author{Fengqiao Luo\thanks{Uber Technologies (\email{luofqfrank@gmail.com})}
\and Jeffrey Larson\thanks{Argonne National Laboratory, Mathematics and Computer Science Division (\email{jmlarson@anl.gov})}
}

\maketitle

\begin{abstract}
  We investigate an empirical quantile estimation approach to solve chance-constrained nonlinear 
  optimization problems. Our approach is based on the 
  reformulation of the chance constraint as an equivalent quantile constraint to provide stronger 
  signals on the gradient. In this approach, the value of the quantile function 
  is estimated empirically from samples drawn from the random parameters, 
  and the gradient of the quantile function is estimated
  via a finite-difference approximation on top of the quantile-function-value estimation. 
  We establish a convergence theory of this approach
  within the framework of an augmented Lagrangian method for solving general nonlinear constrained 
  optimization problems. The foundation of the convergence analysis is a concentration property
  of the empirical quantile process, and the analysis is divided based on whether
  or not the quantile function is differentiable.
  In contrast to the sampling-and-smoothing approach used in the literature,
  the method developed in this paper does not involve any smoothing function
  and hence the quantile-function gradient approximation is easier to implement and there are
  less accuracy-control parameters to tune. 
  We demonstrate the effectiveness of this approach and compare it with a
  smoothing method for the quantile-gradient estimation.
  Numerical investigation shows that the two approaches are
  competitive for certain problem instances. 
\end{abstract}


\section{Introduction}
We investigate methods for solving nonlinear optimization problems with chance constraints:
\beq\label{opt:CCP}
\underset{x\in\ol{\dom}}{\minimize}\;f(x) \;\; \textrm{s.t.: } \P[c_1(x,\xi)\le 0]\ge 1-\alpha,\quad c_2(x)\le 0,  \tag{CCP}
\eeq
where $\xi:\;\Omega\to\R^{n_0}$ is a random parameter with associated probability space 
$(\Omega,\mc{F},P)$, and $f:\mb{R}^n\to\mb{R}$, $c_1:\mb{R}^n\times\mb{R}^{n_0}\to\mb{R}^{l_1}$, and 
$c_2:\mb{R}^n\to\mb{R}^{l_2}$ are three differentiable functions. The parameter $\alpha\in(0,1)$ is a
risk-tolerance value. For example, $\alpha=0.05$ indicating that the constraint $c_1\le{0}$ is allowed to be violated with 
probability at most 0.05.
The closure $\ol{\dom}$ of an open set $\dom$ is an implicit closed subset in $\R^n$ 
for further definitions of regularity conditions, which can be assumed to be sufficiently large.
The constraints in $c_2(x)$ do not involve any random parameters. 
For the sake of clarity, we focus on analyzing the case of $l_1=1$, 
namely the case with a single disjoint chance constraint. 
It is straightforward to generalize 
the analysis and the algorithms presented 
in this work to deal multiple 
disjoint chance constraints. If there are joint chance constraints, our results
can still be applicable provided regularity conditions hold on the constraint functions,
which is discussed in \Cref{sec:diff-joint-cc}.
In this paper $\norm{\cdot}$ denotes the $\ell_2$ norm in the Euclidean space.

\subsection{Literature review}
Optimization with chance constraints was first investigated in~\cite{charnes1958-stoc-prog-heating-oil} 
as a modeling option that requires the probability of an event (formulated by a set of constraints) 
to satisfy a certain threshold. The introducing of chance constraints extends the formulation power in data-driven
decision-making problems. Imposing chance constraints, however, may not preserve convexity of the feasible region,
and hence the chance-constrained programs (CCPs) are very difficult to solve in general if not computationally intractable. 
The reason is that a chance-constraint function is a compound random variable, and the constraint 
complexity is determined by the dependency of the cumulative distribution function 
of the random variable on the decision variables that parameterize it, which could be highly non-convex.
As an important special case of \eqref{opt:CCP}, when the function $c_1(x,\xi)$ can be decoupled as 
$c^\prime_1(x)-\xi$, the probability function $\P[c_1(x,\xi)\le 0]$ can be quasi-concave under certain conditions
(e.g. $\xi$ follows a joint normal distribution), and hence admitting efficient algorithms for the problem 
given that all deterministic functions are convex. The properties of this special case have been investigated in
the seminal work \cite{prekopa1970-prob-constr-prog}.
The differentiability and the formula of derivatives of the probability function are investigated in
\cite{uryasev1994-deriv-prob-fun-integrals,kibzun1995-diff-prob-func,kibzun1998-diff-prob-func,pflug2005-prob-grad-est,garnier2009-asy-form-deriv-prob-func}, especially for the case 
that the random vector follows a multivariate normal distribution 
\cite{henrion2012-grad-form-linear-chance-constr-gauss,ackooij2014-grad-form-chance-constr-guass}.

In the case that the random parameters follow a general probability distribution, 
the finite sample approximation of the continuous (possibly unknown) probability distributions 
is often used as an empirical approach to solving such CCPs in this case. 
Luedtke and Ahmed~\cite{luedtke2008-SAA-opt-chance-constr} 
investigated the effectiveness of a sample-based approximation of chance constraints 
and the sample complexity for achieving certain probabilistic error bounds. 
Additional statistical properties of the sample-based approximation approach to CCPs
have been studied~\cite{calafiore2005-uncert-conv-prog-rand-sol-conf-level,calafiore2006-scen-appr-rob-contrl,calafiore2011-sample-appr-chance-constr-opt,shapiro2005-scen-approx-chance-constr}.
In the scheme of finite-sample approximation or for the case of scenario-based CCPs,
reformulation and decomposition techniques have been developed to establish 
computationally tractable formulations of the CCPs with the help of discrete variables,
when the scenario constraints are convex~\cite{kucukyavuz2014-decomp-alg-two-stage-chance-constr,lodi2022-nonlinear-chance-constr-appl-hydro-sched,luedtke2014-BnC-decomp-alg-chance-constr-finite-supp}.
Valid inequalities have been derived to strengthen the mixed-integer linear program reformulation of a CCP with linear
scenario constraints~\cite{kucukyavuz2012-mix-set-chance-constr-prog,luedtke2010-IP-LP-chance-constr,xie2018-quant-cuts-chance-constr-opt}. Chance constraints have been incorporated into a 
distributionally robust (DR) optimization framework when the information about the underline probability 
distribution is given through a finite set of samples, and tractable reformulations of the DR chance-constrained
optimization have been investigated~\cite{jiang2016-dr-chance-constr-stoc-prog,liu2022-dr-chance-constr-geom-opt,xie2021-dr-chance-constr-wass} under different families
of divergence or distance functions, for example the $\phi$-divergence and the Wasserstein distance. 

For the case of general nonlinear programs with chance constraints induced by general probability distributions, 
a variety of numerical methods based on sampling and smoothing have been developed to
identify high-quality locally optimal
solutions~\cite{zavala2020-sig-approx-chance-constr-nonlinear-prog,curtis2018-seq-alg-solv-nonlinear-opt-chance-constr,geletu2017-inner-outer-approx-chance-constr-opt,luedtke2021-stoc-approx-frontier-chance-constr-nonlinear-program,wachter2020-chance-constr-prob-smooth-sample-nonlinear-approx}.  
Campi and Garatti~\cite{calafiore2011-sample-appr-chance-constr-opt} 
proposed a sampling-and-discarding approach to select a finite set of samples to
induce constraints that are required to be satisfied, and provided sample complexity results
for satisfying the chance constraint with a given probability. 
Along this direction, a primal-dual stochastic gradient method developed in~\cite{xu2020-primal-dual-stoc-grad-conv-func-constrs} can be applied to handle 
the approximated problem with the large number of constraints induced by samples.
(In each iteration of such an approach, only a single constraint is randomly
sampled to compute the gradient of the augmented Lagrangian.)
Geletu et al.~\cite{geletu2017-inner-outer-approx-chance-constr-opt} proposed a smooth approximation
scheme that approximates the chance-constraint function with (smooth) parametric lower and upper estimation
functions and then solves the parametric approximation problems with nonlinear
programming (NLP) solvers.  
Curtis et al.~\cite{curtis2018-seq-alg-solv-nonlinear-opt-chance-constr} developed 
a sequential algorithm for solving the sample approximation problem of a nonlinear CCP, 
in which the outer iteration updates the penalty parameter while the inner sub-problem
is reformulated as a mixed binary quadratic program. 
Kannan and Luedtke~\cite{luedtke2021-stoc-approx-frontier-chance-constr-nonlinear-program}
proposed an approach of constructing the efficient frontier (i.e., solving for optimal objectives 
as the risk value continuously changes) of nonlinear CCPs instead 
of solving the original chance-constrained problem with a pre-specified risk value.
In their formulation, the original chance-constraint function is transformed to be the objective, 
and hence a projected stochastic sub-gradient algorithm~\cite{davis2018-stoc-subgrad-methd-convg-weakly-conv-func} 
can be applied to solve the reformulated problem with smoothing. 
Note that the chance-constraint function has a universal range $[0, 1]$, and hence 
it can be flat in a certain domain, which can slow down the progress of
a gradient-descent method. Motivated by this observation, 
Pe{\~n}a-Ordieres et al.~\cite{wachter2020-chance-constr-prob-smooth-sample-nonlinear-approx}
suggested recasting the chance constraint as a quantile constraint because the later
could have high-magnitude gradients. In their work,
smoothing is applied to the quantile function to achieve numerical stability. 
A CCP can sometimes be equivalently formulated as a 
quantile function optimization problem, and for this case, 
Hu et al.~\cite{hu2022-stoc-approx-sim-quantile-opt} 
proposed a recursive algorithm developed upon the gradient-based
maximum likelihood estimation method~\cite{lam2020-mle-stoch-model} to estimate the gradient
of a differentiable quantile function with respect to the parameter and used it
as an ingredient in a gradient-descent algorithm for minimizing the parameterized quantile function.  
As a special case, Tong et al.~\cite{tong2022-opt-rare-chance-constr} developed 
conservative formulations for NLPs with rare chance constraints induced by 
Gaussian random parameters using tools from large deviation theory.

The algorithm developed in this paper also benefits from the augmented Lagrangian
methods developed for deterministic nonlinear optimization with equality and inequality
constraints \cite{birgin-practical-aug-lag,birgin2012-bd-penalty-param-aug-lag-ineq-constr,birgin2005-num-compare-aug-lag-alg}.
In particular, we established a probabilistic augmented Lagrangian method to handle
quantile constraint(s) that are only accessible based on samples.

\subsection{Contributions} Inspired
by~\cite{wachter2020-chance-constr-prob-smooth-sample-nonlinear-approx}, our approach 
reformulates the chance constraint into a quantile constraint. We develop 
a derivative-free approach to handle the evaluation of the quantile-function gradients.
Specifically, we use a sample-approximation method to estimate the value of the quantile function
based on a concentration theory of the empirical process. We also use a finite-difference
approach to estimate the gradients of the quantile function, 
which are used to build local approximation models.
Local models are used in an augmented
Lagrangian method (ALM) to solve the NLP with quantile constraints and other 
deterministic nonlinear constraints, while the inner optimization problems (with given values
of penalty parameters) are solved using a trust-region method. We have developed a 
convergence theory when the quantile function is differentiable.

\section{Quantile constraint reformulation}\label{sec:quant-reform}
Let $\Xi_x \defined c_1(x,\xi)$ be a random variable parameterized by $x$
and the random vector $\xi$.
As noted in~\cite{wachter2020-chance-constr-prob-smooth-sample-nonlinear-approx},
the chance constraint $\P[c_1(x,\xi)\le0]\ge 1-\alpha$ is equivalent to 
\beq
Q^{1-\alpha}(x)\le 0,
\eeq
where $Q^{1-\alpha}(x)$ is the $1-\alpha$ quantile of $\Xi_x$. 
If the $1-\alpha$ quantile of $\Xi_x$ is not unique,
we set $Q^{1-\alpha}(x) \defined \inf_{q\in\mc{Q}^{1-\alpha}(x)} q$, 
where $\mc{Q}^{1-\alpha}(x)$ is the set of all $1-\alpha$ quantiles of $\Xi_x$.
With the reformulation of the chance constraint, 
\eqref{opt:CCP} can be reformulated as 
the quantile-constrained problem
\beq\label{opt:QCP}
\underset{x\in\R^n}{\minimize}\;f(x) \;\; \textrm{s.t.: } \Qx\le 0,\quad c_2(x)\le 0. \tag{QCP}
\eeq
We assume that the gradients of the functions
$f(\cdot)$ and $c_2(\cdot)$ are accessible at any $x\in\mb{R}^n$,
and i.i.d.~samples of $\xi$ can be drawn as needed.
Note that the essential difficulty of \eqref{opt:QCP}
is that the derivative of the quantile function $Q^{1-\alpha}(\cdot)$ 
is not available in general although $\nabla c_1$ is computable.  
Furthermore, a sampling method is needed in order to obtain zeroth-order information 
about (i.e., the value of) $Q^{1-\alpha}(\cdot)$.
This paper develops an approach to find a stationary point
of \eqref{opt:QCP} using sample-based estimators of $Q^{1-\alpha}(\cdot)$ 
and $\nabla Q^{1-\alpha}(\cdot)$. 
These approximations are incorporated in an augmented Lagrangian method 
to solve the constrained problem \eqref{opt:QCP}. 

We first focus on the case where $\Qx$ is continuously differentiable 
within a bounded domain (\Cref{ass:c1-diff,ass:Q-diff}).
We also assume common
regularity conditions (\Cref{ass:f-c2-diff}) hold for the functions $f$, $c_1$ and $c_2$:
\begin{assumption}\label{ass:f-c2-diff}
The objective $f$ and the constraint functions $c_2$ 
are continuously differentiable at any $x\in\dom$,
and their gradients are Lipschitz continuous.
The corresponding Lipschitz constants are denoted as $L_f$
and $L_{c_2}$.
\end{assumption}
\begin{assumption}\label{ass:c1-diff}
For any $x\in\dom$, the random variable $c_1(x,\xi)$ (or $\Xi_x$) has a continuously
differentiable probability density function w.r.t. the Lebesgue measure.
The inverse function $F^{-1}_{\Xi_x}$ exists and it is nonzero in a neighborhood
of $F^{-1}_{\Xi_x}(1-\alpha)$, where $F_{\Xi_x}$ is the cumulative distribution function
of $\Xi_x$ for any $x\in\dom$.  
\end{assumption}
\begin{assumption}\label{ass:Q-diff}
The quantile function $Q^{1-\alpha}(\cdot)$ is continuously differentiable at any $x\in\mb{R}^n$,
and $\nabla Q^{1-\alpha}(\cdot)$ is Lipschitz continuous with the Lipschitz constant $L_Q$.
\end{assumption} 
\begin{assumption}\label{ass:F_twice_diff}
Assume that for every $x\in\ol{\dom}$, 
the cumulative density function $F_{\Xi_x}(\cdot)$ is twice differentiable in $(a_x, b_x)$,
where $a_x \defined \sup\{t: F_{\Xi_x}(t)=0\}$ and $b_x \defined \inf\{t: F_{\Xi_x}(t)=1\}$.
\end{assumption}
Note that verifying the differentiability of a quantile function with respect
to the decision variable can be difficult in practice. Therefore,
\Cref{ass:Q-diff} requires some justification: Sufficient conditions for
\Cref{ass:Q-diff} to hold are discussed in~\cite{hong2009-est-quant-sens}. 
To ensure this paper is self-contained, these conditions are stated with the notation of this paper in 
\Cref{ass:DCC}. \Cref{ass:Q-diff} is
implied by \Cref{ass:DCC}, but the later may appear to
be easier to verify. 

\begin{assumption}\label{ass:DCC}
  The disjoint chance constraint satisfies the following:
\begin{enumerate}
	\item The sample-wise gradient $\nabla_x c_1(x, \xi)$ exists w.p.1 for any $x\in \dom$, 
	and there exists a function $k(\xi)$ with $\E[k(\xi)]<\infty$ such that 
  $|c_1(x_1,\xi)-c_1(x_2,\xi)|\le k(\xi)\norm{x_1-x_2}$ for all $x_1,x_2\in \dom$.\label{subass:1}
	\item For any $x\in \dom$, the random variable $c_1(x,\xi)$ has a continuous 
	density function $\rho(t; x)$ and $\nabla_x F(t; x)$ exists and is continuous 
	with respect to both $x$ and $t$, where $F(t;x)$ is the cumulative distribution
  function of $c_1(x,\xi)$.\label{subass:2}
	\item For any $x\in \dom$, the function $g(t;x) \defined \E[\nabla_x c_1(x,\xi) | c_1(x,\xi)=t]$ 
    is continuous with respect to $t$.\label{subass:3}
\end{enumerate}
\end{assumption}
We will denote the three items in \Cref{ass:DCC} as \Cref{ass:DCC}($i$) for $i=1,2,3$, 
respectively from hereafter. 
We note that the algorithms in this paper might also be applied to some 
CCPs for which the above assumptions are not necessarily satisfied. 
These assumptions are for convergence analysis.

\subsection{Differentiability of the quantile function of a joint chance constraint}
\label{sec:diff-joint-cc}
For the case of having a joint chance constraint, it requires additional procedures
to reformulate the chance constraint into a quantile constraint. 
In this case the definition of the quantile function should be adjusted. 
Specifically, consider the following joint chance constraint:
\beq\label{eqn:joint-cc}
\P[\psi_i(x,\xi)\le 0,\;i=1,\ldots,l]\ge 1-\alpha.
\eeq
We define the function $\psi$ as 
$\psi(x,\xi) \defined \max\{\psi_i(x,\xi),\;i=1,\ldots,l\}$, which is a random variable
parameterized by $x$. Ideally, the joint chance constraint \eqref{eqn:joint-cc}
can be reformulated as a quantile constraint $Q^{1-\alpha}_\psi(x)\le 0$,
where $Q^{1-\alpha}_\psi(x)$ denotes the $1-\alpha$ quantile of the random
variable $\psi(x,\xi)$. The following proposition provides conditions under
which the quantile function $Q^{1-\alpha}_\psi(\cdot)$ is differentiable and
hence our convergence analysis and algorithms can be applied
to joint chance constraints. 
\begin{proposition}\label{prop:joint-Q-diff}
If the following conditions hold
\begin{enumerate}
  \item Each $\psi_i$ satisfies \subassref{DCC}{1},\label{jcc1}
	\item For any $x\in \dom$, the random vector $[\psi_1(x,\xi),\ldots,\psi_l(x,\xi)]$ has a continuous joint 
		 density function $\rho(t_1,\ldots,t_l;x)$ and $\nabla_x F(t_1,\ldots,t_l;x)$ exists and is continuous
		 with respect to $x$ and all $t_i$, where $F(t_1,\ldots,t_l;x)$ is the cumulative distribution function
     of $[\psi_1(x,\xi),\ldots,\psi_l(x,\xi)]$, and\label{jcc2}
	\item For any $x\in \dom$, the function 
	$g_i(t_1,\ldots t_l; x) \defined \E[\nabla_x\psi_i(x,\xi)|\psi_j(x,\xi)=t_j, \forall j]$
  is continuous with respect to $t$ for every $i\in[l]$, where $[l]:=\{1,\ldots,l\}$. \label{jcc3}
\end{enumerate}
then the quantile function $Q^{1-\alpha}_\psi(x)$ is differentiable for all $x\in \dom$.
\end{proposition}
\begin{proof}[Proof of \Cref{prop:joint-Q-diff}]
It suffices to show that $\psi$ satisfies the conditions in \Cref{ass:DCC}. 
First, we note that condition~\ref{jcc2} implies $\P[\psi_i(x,\xi)=\psi_j(x,\xi)]=0$ for all $i<j$.
Notice that $\nabla_x\psi(x,\xi)=\nabla_x \psi_i(x,\xi)$ for $\xi$ in
the region $\{\xi:\;\psi_i(x,\xi)>\psi_j(x,\xi)\;\forall j\in[l]\setminus\{i\}\}$
for all $i\in[l]$, and since $\P[\psi_i(x,\xi)=\psi_j(x,\xi)]=0$ for all $i<j$ 
(due to the continuity of the probability density function),
the gradient $\nabla_x\psi(x,\xi)$ exists w.p.1. Furthermore,
$|\psi(x_1,\xi)-\psi(x_2,\xi)|\le\max_i\{k_i(\xi)\}\norm{x_1-x_2}$,
where $k_i(\xi)$ is the function in \Cref{ass:DCC}(1) for $\psi_i$
(using the condition 1 of \Cref{prop:joint-Q-diff}).
This verifies that \Cref{ass:DCC}(1) is satisfied by $\psi$.
To verify that $\psi$ satisfies \Cref{ass:DCC}(2), we see that
$\psi(x,\xi)$ has a continuous density function if the cumulative distribution
function $F_\psi(t;x)$ of it is differentiable with respect to $t$. 
The $F_\psi(t;x)$ can be computed as
\beq
\ba
F_\psi(t;x)&=\P[\psi(x,\xi)\le t] 
=\P[\psi_1(x,\xi)\le t,\ldots, \psi_l(x,\xi)\le t] =F(t,\ldots,t; x).
\ea
\eeq
It follows that
\beq
\ba
&\frac{dF_\psi(t;x)}{dt}=\sum^l_{i=1}\partial_i F(t,\ldots,t;x) \\
&=\sum^l_{i=1}\int^t_{-\infty}\ldots\int^t_{-\infty} \rho(t_1,\ldots,t_{i-1},t,t_{i+1},\ldots,t_l; x)
dt_1\ldots dt_{i-1}dt_{i+1}\ldots dt_l,
\ea
\eeq
which is clearly continuous in $t$. 
The gradient $\nabla_x F_\psi(t;x)=\nabla_x F(t,\ldots,t;x)$
which is continuous in $x$ and $t$ by condition~\ref{jcc2}. 
To verify that $\psi$ satisfies Assumption 4.3, we notice that
for any $x\in \dom$, the function $g(t;x)$ can be computed as
\beq\label{eqn:g_psi}
\ba
g(t;x)&=\E[\nabla_x \psi(x,\xi) | \psi(x,\xi)=t] \\
&=\sum^l_{i=1}\E[\nabla_x\psi_i(x,\xi)|\psi_i(x,\xi)=t,\;\psi_i(x,\xi)\ge\psi_j(x,\xi)\; \forall j\in[l]\setminus\{i\}].
\ea
\eeq
Note that the second inequality of \eqref{eqn:g_psi} holds since 
we have $\P[\psi_i(x,\xi)=\psi_j(x,\xi)]=0$ for all $i\neq j$. For $i=1$,
the first function in the summation can be represented as the following integral
\beq\label{eqn:app-cond-exp}
\ba
&\E[\nabla_x\psi_1(x,\xi)|\psi_1(x,\xi)=t,\;\psi_1(x,\xi)\ge\psi_j(x,\xi)\; \forall j\ge 2] \\
&=\int^t_{-\infty}\ldots\int^t_{-\infty} g_1(t,t_2,\ldots t_l; x) dt_2\ldots dt_l,
\ea
\eeq
which exists by condition~\ref{jcc3}. Note that in \eqref{eqn:app-cond-exp} 
the conditional expectation on $\psi_1(x,\xi)=t,\;\psi_1(x,\xi)\ge\psi_j(x,\xi)\; \forall j\ge 2$
has been achieved by setting the first argument in $g_1$ as $t$ for achieving $\psi_1(x,\xi)=t$ 
and setting the range of integral to be from $-\infty$ to $t$ for all the other arguments. 
Now we have verified that all the three conditions in \Cref{ass:DCC} are satisfied by the 
function $\psi(x,\xi)$, and hence the quantile function $Q^{1-\alpha}_\psi(x)$ is differentiable for all $x\in \dom$.
\qed
\end{proof}

\section{Sample-Based Quantile Approximation}
\label{sec:quant-grad-approx}
Our approach for solving \eqref{opt:QCP} uses
estimates $\Qx$ and $\nabla\Qx$ at a given point $x$ with samples of $\xi$,
and a finite-difference approximation is used for the estimation of $\nabla\Qx$.
Consider the problem of estimating the $1-\alpha$ quantile of a random variable $X$
using $N$ i.i.d.~samples $\{X_i\}^N_{i=1}$ drawn from the probability distribution of $X$.
Denote the quantile as $X^{(1-\alpha)}$. An asymptotic unbiased estimator of $X^{(1-\alpha)}$
is the $1-\alpha$ quantile $\wh{X}\up{1-\alpha}$ of the sequence $\{X_i\}^N_{i=1}$, 
which is computed by 
\begin{enumerate}
	\item Sorting $\{X_i\}^N_{i=1}$ in ascending order, 
		and letting $\{X^{\prime}_i\}^N_{i=1}$ be the sorted sequence;
	\item Letting $\wh{X}\up{1-\alpha}$ be the $\ceil{(1-\alpha)N}$th element in the sorted sequence.
\end{enumerate} 
The following theorem connects the quantile process with a Brownian bridge.
\begin{theorem}[{\cite[Theorem 6]{gsorgo1978-strong-approx-quant-process}}]\label{thm:quant-proc}
Let $\{X_i\}^N_{i=1}$ be i.i.d.~random variables with a cumulative distribution
function $F_X$ that is 
twice differentiable on $(a,b)$ where $a \defined \sup\{x: F_X(t)=0\}$, $b \defined \inf\{x: F_X(t)=1\}$
($a$ and $b$ can be $-\infty$ and $+\infty$ respectively in general)
and the density function $\rho_X \defined F^{\prime}_X$ is continuously differentiable and greater than zero on $(a,b)$. 
Let $N$ be the sample size and 
$\hat{q}^{1-\alpha}_N \defined N^{\frac{1}{2}}(\wh{Q}^{1-\alpha}_N-Q^{1-\alpha})$,
where $\wh{Q}^{1-\alpha}_N$ is the empirical $(1-\alpha)$-quantile 
of the sample set $\{X_i\}^N_{i=1}$ determined by the above two steps
and $Q^{1-\alpha} \defined F^{-1}_X(1-\alpha)$ is the unique $(1-\alpha)$-quantile of the distribution $F_X$. 
The following concentration inequality holds for every $z>0$:
\beq
\P\left(\big|\rho_X(Q^{1-\alpha})\hat{q}^{1-\alpha}_N-\sqrt{\alpha(1-\alpha)}W\big|>(C\log N+z)\log N\right)< e^{-d z},
\eeq
where $W$ is a standard normal random variable, $C$ and $d$ are constants that depend on $|\rho_X(Q^{1-\alpha})|$ and $|\rho^\prime_X(Q^{1-\alpha})|$, where $\rho^\prime_X$ is the derivative of the density function $\rho_X$.
\end{theorem}
We can now provide a probabilistic bound on the quantile estimation's gradient.
\begin{theorem}\label{thm:quant-grad-approx}
Suppose \Cref{ass:c1-diff,ass:Q-diff,ass:F_twice_diff} hold. Let $\rho(q;x)$ and $\rho^\prime(q;x)$ 
be the density function and the derivative of the density function of $\Xi_x$ evaluated at $q$.   
Let $\{\Xi_i\}^N_{i=1}$ be i.i.d.~samples of $\Xi_x$ and
$\Q{N}{x}$ be the empirical $1-\alpha$ quantile of the samples. 
Consider the sample-based quantile gradient estimator $\wh{D}_N(x)$ defined by
\beq\label{eqn:whG}
\wh{D}_N(x) \defined \sum^n_{k=1}\frac{\Q{N}{x+\beta\hat{e}_k}-\Q{N}{x-\beta\hat{e}_k}}{2\beta}\hat{e}_k,
\eeq
where $\beta>0$ is the finite-difference parameters, and $\hat{e}_k$ is the basis vector in 
$\mR^n$ with a one in entry $k$ and zeros in all the other entries.
If $\beta\le \frac{\delta}{nL_Q}$ and the sample size $N$ satisfies the following condition:
\beq\label{eqn:N>lb1}
N\ge O\left( \frac{C^2_\dom\alpha(1-\alpha)n^2(\log\frac{n}{\gamma})^3}{d^2_\dom\beta^2\delta^2} \right),
\eeq
the following probabilistic bound on the gradient of the quantile estimation holds
for every $\delta,\gamma>0$ and $x\in\ol{\dom}$:
\beq\label{eqn:N-cond1}
\P\parenth{\norm{\nabla\Qx-\wh{D}_N(x)}\ge\delta} \le \gamma,
\eeq 
where $C_{\dom}$ and $d_\dom$ are positive constants that depend on the bounds on
$\rho(\Qv{x};x)$ and $\rho^\prime(\Qv{x};x)$ for all $x\in\ol{\dom}$.
\end{theorem}
\begin{proof}
The estimation error of the quantile gradient can be bounded as follows
\[
\ba
&\norm{\nabla\Qx-\wh{D}_N(x)} 
\le\sum^n_{k=1}\abs{\partial_k\Qx-\frac{\Q{N}{x+\beta\hat{e}_k}-\Q{N}{x-\beta\hat{e}_k}}{2\beta}} \\
&\le\sum^n_{k=1}\abs{\partial_k\Qx-\frac{Q^{1-\alpha}(x+\beta\hat{e}_k)-Q^{1-\alpha}(x-\beta\hat{e}_k)}{2\beta}} \\
&\quad+\sum^n_{k=1}\abs{\frac{Q^{1-\alpha}(x+\beta\hat{e}_k)-Q^{1-\alpha}(x-\beta\hat{e}_k)}{2\beta}
-\frac{\Q{N}{x+\beta\hat{e}_k}-\Q{N}{x-\beta\hat{e}_k}}{2\beta}} \\
&\le \frac{1}{2}nL_Q\beta + \frac{1}{2\beta}\sum^n_{k=1}
\abs{Q^{1-\alpha}(x-\beta\hat{e}_k)-\Q{N}{x-\beta\hat{e}_k}} \\
&\quad +\frac{1}{2\beta}\sum^n_{k=1}\abs{Q^{1-\alpha}(x+\beta\hat{e}_k)-\Q{N}{x+\beta\hat{e}_k}},
\ea
\]
where $\partial_k$ denotes the derivative with respect to the $k$th component of $x$,
and the first term $\frac{1}{2}nL_Q\beta$ in the last inequality is due to \Cref{ass:Q-diff} on the $L_Q$-smoothness
of the $Q^{1-\alpha}$ function.
We choose $\beta \defined \frac{\delta}{nL_Q}$, which yields the following inequalities using 
the union bound:
\beq\label{eqn:P(Q-G)}
\ba
\P\parenth{\norm{\nabla\Qx-\wh{D}_N(x)}\ge\delta} 
&\le\P\parenth{\frac{1}{2\beta}\sum^n_{k=1}\parenth{T_{1k} + T_{2k}}\ge\frac{\delta}{2}} \\
&\le\sum^n_{k=1}\P\parenth{T_{1k}\ge\frac{\beta\delta}{2n}}
+\sum^n_{k=1}\P\parenth{T_{2k}\ge\frac{\beta\delta}{2n}},
\ea
\eeq
where $T_{1k}$ and $T_{2k}$ are defined as 
\begin{displaymath}
\begin{aligned}
& T_{1k}  \defined  \abs{Q^{1-\alpha}(x-\beta\hat{e}_k)-\Q{N}{x-\beta\hat{e}_k}}, \\
& T_{2k}  \defined  \abs{Q^{1-\alpha}(x+\beta\hat{e}_k)-\Q{N}{x+\beta\hat{e}_k}}.
\end{aligned}
\end{displaymath}
Applying \Cref{thm:quant-proc}, for a fixed $x_0$ and any fixed sample size $N$ 
the following inequality holds with probability at least $1-e^{-d_{x_0} y}$:
\beq
\Big|z_{x_0}N^{\frac{1}{2}}(\Q{N}{x_0}-\Qv{x_0})-\sqrt{\alpha(1-\alpha)}W_{x_0}\Big|\le(C_{x_0}\log N+y)\log N,
\eeq
where $z_{x_0} \defined \rho(\Qv{x_0};x_0)$ is the density function of $\Xi(x_0)$ evaluated at the 
quantile $\Qv{x_0}$. Note that the constant $d_{x_0}$, $C_{x_0}$ and the standard normal
variable $W_{x_0}$ are all $x_0$ dependent, and the parameter $y$ will be chosen later. 
It leads to the following inequalities 
with probability at least $1-e^{-d_{x_0} y}$:
\beq\label{eqn:Q-Q_bd1}
\begin{aligned}
&\Q{N}{x_0}-\Qv{x_0} \le \sqrt{\alpha(1-\alpha)}z^{-1}_{x_0}N^{-1/2}W_{x_0} + z^{-1}_{x_0}N^{-1/2}(C_{x_0}\log N+y)\log N, \\
&\Q{N}{x_0}-\Qv{x_0} \ge \sqrt{\alpha(1-\alpha)}z^{-1}_{x_0}N^{-1/2}W_{x_0} - z^{-1}_{x_0}N^{-1/2}(C_{x_0}\log N+y)\log N
\end{aligned}
\eeq
We now impose the following conditions to bound each term in \eqref{eqn:Q-Q_bd1}:
\beq\label{eqn:Q-Q_term_bd}
\begin{aligned}
&\P\left(\sqrt{\alpha(1-\alpha)}z^{-1}_{x_0}N^{-1/2}W_{x_0}\ge\frac{\beta\delta}{6n}\right)<\gamma^\prime, \\
&C_{x_0}z^{-1}_{x_0}N^{-1/2}(\log N)^2<\frac{\beta\delta}{6n}, \\
&C_{x_0}yz^{-1}_{x_0}N^{-1/2}\log N<\frac{\beta\delta}{6n},
\end{aligned}
\eeq
where $\gamma^\prime$ is a parameter which will be chosen later.
Using the tail bound of the Gaussian variable, the above three inequalities 
are guaranteed to hold if the sample size $N$ should satisfy
\beq\label{eqn:sp_N_cond1}
N\ge O\left(\frac{\alpha(1-\alpha)n^2C^2_{x_0}y^2\log\frac{1}{\gamma^\prime}}{\beta^2\delta^2z^2_{x_0}} \right).
\eeq
Combining \eqref{eqn:Q-Q_bd1} and \eqref{eqn:Q-Q_term_bd} concludes that 
\beq\label{eqn:P<e+gamma}
\P\parenth{\abs{Q^{1-\alpha}(x_0)-\Q{N}{x_0}}\ge\frac{\beta\delta}{2n}}<e^{-d_{x_0}y}+\gamma^\prime.
\eeq
Then we substitute $x_0\gets x\pm \beta\hat{e}_k$ for $k=1,\ldots,n$ and incorporate them into \eqref{eqn:P(Q-G)}.
We then obtain the following inequality by applying the union bound if \eqref{eqn:sp_N_cond1}
and \eqref{eqn:P<e+gamma} are satisfied 
with universal constants $C_{\dom}$ and $d_\dom$ that depend on the uniform bounds for
$\rho(\Qv{x};x)$ and $\rho^\prime(\Qv{x};x)$ over $x\in\ol{\dom}$:
\beq\label{eqn:|Q-D|_interm}
\P\parenth{\norm{\nabla\Qx-\wh{D}_N(x)}\ge\delta} \le 2n(e^{-d_\dom y}+\gamma^\prime).
\eeq
We set $y=\frac{1}{d_\dom}\log\frac{4n}{\gamma}$ and $\gamma^\prime=\frac{\gamma}{4n}$,
the right-hand side of \eqref{eqn:|Q-D|_interm} becomes $\gamma$, and the complexity condition becomes \eqref{eqn:N>lb1}.
\qed
\end{proof}
The following corollary is directly implied from the proof of \Cref{thm:quant-grad-approx}.
\begin{corollary}\label{cor:P(|Q-QN|)}
Suppose \Cref{ass:c1-diff,ass:Q-diff,ass:F_twice_diff} hold.
If $N\ge O\left( \frac{C^2_\dom\alpha(1-\alpha)(\log\frac{1}{\gamma})^3}{d^2_\dom\delta^2} \right)$,
the following inequality holds for all $x\in \dom$:
\beq\label{eqn:P(|Q-Q|)<gamma}
\P\parenth{\abs{\Qx-\Q{N}{x}}\ge\delta}\le\gamma.
\eeq
\end{corollary}
Note that compared to \eqref{eqn:N>lb1} the sample complexity in the above corollary 
is independent of $n$ and $\beta$ since it is for the estimation of the quantile itself rather than
the gradient of the quantile function (which involves the step size $\beta$ and dimension $n$). 
Clearly, the sample complexity for \eqref{eqn:P(|Q-Q|)<gamma} to hold is dominated by 
\eqref{eqn:N>lb1}.

\section{Stochastic Merit Function Based on the Quantile Approximation}
\label{sec:aug-lag}
We now analyze an augmented Lagrangian method (ALM)
to solve \eqref{opt:QCP}.
To simplify the analysis, we unify the notation for constraint functions and reformulate \eqref{opt:QCP} as 
\beq\label{opt:NLP}
\underset{x\in\R^n}{\minimize}\;f(x) \;\; \textrm{s.t. } g_i(x)\le 0 \quad\forall i\in\I_0, \tag{NLP}
\eeq
where $\I_0 \defined \{0,1,\ldots,l_2\}$, $\I=\{1,\ldots,l_2\}$, $g_0(x):=Q^{1-\alpha}(x)$, 
and $g_i(x) \defined c_{2,i}(x)$ for $i\in[l_2]$. 
\Cref{def:constr-qualify} restates the
two standard constraint qualifications based on problem setting of \eqref{opt:NLP},
which will be used for characterizing its optimality conditions.
\begin{definition}\label{def:constr-qualify}
(1) The Mangasarian-Fromovitz constraint qualification (MFCQ) condition \cite{MFCQ1967} 
is satisfied by~\eqref{opt:NLP} at a point $x^*$ if there exists $d\in\mR^n$ 
such that $\inner{\nabla g_i(x^*)}{d}<0$ for all $i\in\mc{A}(x^*)$, where 
$\mc{A}(x^*):=\{i\in\I_0: g_i(x^*)=0\}$ is the set of active constraints at $x^*$.
\newline
(2) The constant positive linear independence (CPLD) condition \cite{CPLD2000} 
is satisfied by~\eqref{opt:NLP} at a point $x^*$, if there exists a neighborhood $U$
of $x^*$ such that for any $J\subseteq\I_0$ the set of gradients $\{\nabla g_i(x^*):\;i\in J\}$ are linearly dependent,
$\{\nabla g_i(x): i\in J\}$ are linearly dependent for all $x\in U$. 
\end{definition}

Given $x\in\mR^n, \mu\in\mR^{|\I_0|}_+, \rho>0$, we define the 
Powell-Hestenes-Rockafellar augmented Lagrangian \cite{birgin2005-num-compare-aug-lag-alg} by
\beq\label{eqn:aug-Lag}
\Phi(x,\rho,\mu) \defined f(x)+\frac{\rho}{2}\sum_{i\in\I_0}\max\left\{0, g_i(x)+\frac{\mu_i}{\rho}\right\}^2.
\eeq
The augmented Lagrangian \eqref{eqn:aug-Lag} will be use as the 
merit function in the augmented Lagrangian method for solving \eqref{opt:NLP},
and it is initially proposed for deterministic constrained optimization problem with inequality constraints.
For fixed parameters $\rho,\mu$, the gradient of $\Phi(x,\rho,\mu)$ is given by
\beq\label{eqn:grad-Phi}
\nabla\Phi(x,\rho,\mu) = \nabla{f(x)}+\rho\sum_{i\in\I_0}\max\bigcurl{0, g_i(x)+\frac{\mu_i}{\rho}}\nabla{g_i(x)}.
\eeq
Because the analytical form of $\nabla{g_0(x)}$ is assumed not to be available, 
we use the following estimator to approximate $\nabla\Phi(x,\lambda,\mu)$:
\beq\label{eqn:phi_N}
\ba
\phi_N(x,\rho,\mu,\beta) &\defined \nabla{f(x)}
+\rho\max\bigcurl{0, g_{0,N}(x)+\frac{\mu_i}{\rho}}G_N(x) \\
&\quad\;+\rho\sum_{i\in\I}\max\bigcurl{0, g_i(x)+\frac{\mu_i}{\rho}}\nabla{g_i(x)},
\ea
\eeq
where $\gN(x) \defined \Q{N}{x}$ and $G_N(x)$ is an estimator 
of $\nabla{g_0(x)}$ defined as
\beq
G_N(x) \defined \sum^n_{j=1}\frac{\Q{N}{y+\beta\hat{e}_j}-\Q{N}{y-\beta\hat{e}_j}}{2\beta}\hat{e}_j.
\eeq

We can build a local model to approximate the merit function in a small neighborhood 
$B(x_0,\Delta)$ of a point $x_0$. This local model can utilize the estimated
zeroth-order, first-order, and (when available) second-order information. 
The local model has the form 
\beq\label{eqn:m_N}
m_N(x,\rho,\mu,\beta)  \defined  \Phi_N(x_0,\rho,\mu) + \phi_N(x_0,\rho,\mu,\beta)^{\top}(x-x_0)
+\frac{1}{2}(x-x_0)^{\top}H(x-x_0),
\eeq
where $\phi_N$ is as in \eqref{eqn:phi_N} and $\Phi_N(x_0,\rho,\mu)$ is the $N$-sample approximation of $\Phi(x_0,\rho,\mu)$:
\beq\label{eqn:Phi_N}
\Phi_N(x_0,\rho,\mu) = f(x_0)+\frac{\rho}{2}\max\left\{0, g_{0,N}(x_0)+\frac{\mu_0}{\rho}\right\}^2
+\frac{\rho}{2}\sum_{i\in\I}\max\left\{0, g_i(x_0)+\frac{\mu_i}{\rho}\right\}^2
\eeq
The parameters $\beta$, $N$, $\Delta$, and the matrix $H$ 
control the accuracy of the approximation.
Note that in the algorithm analysis, we do not require additional
conditions on $H$ other than that its norm is bounded. 
In the implementation, we use a neighborhood
sampling approach to build a local quadratic model and extract the Hessian
of the local model as an approximation of $H$. See \Cref{sec:num} for further details.
We also define the quantity
\beq
\sigma_N(x,\mu)=\sqrt{\min\{-g_{0,N}(x),\mu_0\}^2+\sum_{i\in\I}\min\{-g_i(x),\mu_i\}^2},
\eeq
which will be used to quantify the level of feasibility. 
\section{Algorithms}
\label{sec:alg}
We now discuss our approach for optimizing \eqref{opt:NLP} where the problem
is decomposed into an outer problem and an inner problem that are solved
repeatedly. 
The inner problem is to minimize the sample-based merit function 
for fixed penalty parameter $\rho$ and Lagrangian multipliers $\mu$:
\beq\label{opt:inner}
\minimize_x \Phi(x,\rho,\mu).
\eeq 
Since the analytical form of quantile function is unknown, the $\Phi$
is approximated by a sample-based estimation $\Phi_N$ (defined in \eqref{eqn:Phi_N}) 
in each iteration for solving the inner problem with an increasing sample size $N$. 
We seek a local minimizer for the inner problem with this process being terminated 
once a certain criterion is met. In the outer problem, the ALM is applied to update $\rho$ and $\mu$. 

\begin{algorithm2e}[t]
  \fontsize{8}{8}\selectfont
  \SetAlgoNlRelativeSize{-4}
  \KwData{Initial point $x_{\mathrm{in}}$ and parameters $\rho,\mu,r,\varepsilon$
  such that $\mu\ge0$, $\rho>0$ and $0<r,\varepsilon<1$.}
 \textbf{Internal parameter:} Pick $0<\derate<1<\inrate$, $0<\eta_1,\eta_2,r_0<1$, $0<\Delta_0$. 
 
 \For{$k=0,1,\ldots$}{
 Choose $\beta \gets r_0\Delta_k$ for each iteration.
 Build a $(1-\varepsilon)$-probabilistically $\ck$-fully linear model $m_{N_k}$ (See \Cref{def:fully-linear-model})
 on $B(x^k,\Delta_k)$ with sample size $N_k$ satisfying \eqref{eqn:N_cond_kappa_alpha}

 Compute $s^k \gets \underset{s:\|s\|\le\Delta_k}{\textrm{arg min}}\;\; m_{N_k}(x^k+s,\rho,\mu,\beta)$

  \eIf{$m_{N_k}(x^k,\rho,\mu,\beta)-m_{N_k}(x^k+s^k,\rho,\mu,\beta)\ge\eta_1\emph{min}\{\Delta_k, \Delta^2_k\}$}{
    Calculate $\rho_k \gets \frac{\Phi_{N_k}(x^k,\rho,\mu)-\Phi_{N_k}(x^k+s^k,\rho,\mu)}{m_{N_k}(x^k,\rho,\mu,\beta)-m_{N_k}(x^k+s^k,\rho,\mu,\beta)}$ 

    \eIf{$\rho_k\ge\eta_2$}
    	{
		$x^{k+1} \gets x^k+s^k$;  $\Delta_{k+1} \gets \inrate\Delta_k$
    	}
    	{
		$x^{k+1} \gets x^k$;  $\Delta_{k+1} \gets \derate\Delta_k$
    	}
   }
   {
   $x^{k+1} \gets x^k$; $\Delta_{k+1} \gets \derate\Delta_k$
  }
  \eIf{$\Delta_{k+1}\le r$}
  { \textbf{return} $x^k$.}
  {Continue}
  }
  \caption{\fontsize{8}{8}\selectfont Trust-region method for merit functions with penalized chance constraints\label{alg:trust-region-merit-func}}
\end{algorithm2e}

For solving the inner problem \eqref{opt:inner},
we apply a trust-region algorithm (\Cref{alg:trust-region-merit-func}) with probabilistic ingredients.
It is similar to a traditional trust-region algorithm, which enlarges or shrinks
the trust region based on the relative improvement ratio. Of course, because the 
quantile function's value and gradient 
are based on samples approximation,
the algorithmic procedures (i.e., decisions on enlarging or shrinking the trust region
and the termination criteria) built on top of such information are all stochastic. 
Probabilistic trust-region algorithms have been developed (e.g., in~\cite{larson2014-deriv-free-stoch-opt})
for derivative-free unconstrained optimization problems where the form of the objective is unknown and
its value can be approximated only from sampling.
In contrast to the problems studied in~\cite{larson2014-deriv-free-stoch-opt},
in our objective $\Phi_N$, only the form
of the quantile function is assumed unavailable and must be estimated with a probabilistic model.

A probabilistic ALM for the outer problem is given as \Cref{alg:penalty}.
This algorithm updates the Lagrangian multipliers $\mu$ 
and the penalty parameter $\rho$ in \eqref{eqn:params-update}. 
\begin{algorithm2e}[t]
  \DontPrintSemicolon
    \fontsize{8}{8}\selectfont
  \SetAlgoNlRelativeSize{-4}
  \SetKw{true}{true}
\KwData{List of input parameters: $\theta_r,\theta_{\rho},\varepsilon\in(0,1),\mumax>0$, 
$\theta_{\rho}>1$, and convergence control parameters $\{r^k,\eta^k\}^\infty_{k=1}$ 
satisfying $\eta^k\to0$.}
Set $\mu_i=\bar{\mu}_i=\mu_{init}$ for some $\mu_{init}>0$ and $i\in\I_0$, and initialize $x^0$. \;
Set $\rho^1=\rho_{init}$ for some $\rho_{init}>0$. \;
\For{$k=1,\ldots$}{

\textbf{Step 1.~merit function minimization:} Apply \Cref{alg:trust-region-merit-func}
with the initial point $x^{k-1}$ obtained from the previous iteration when $k{\ge}1$ 
and with the input parameters $\rho^k,\bar{\mu}^k,r^k,\varepsilon$ to get a termination point 
$x^k$ as the output.\;

\textbf{Step 2.~convergence-control parameters reduction:}
Generate $N^k$ independent samples (The conditions of $N^k$ to guarantee convergence
will be given in Section~\ref{sec:convg-analy}.), 
evaluate $g_{0,N^k}(x^k)$, and update
\beq\label{eqn:params-update}
\ba
&\mu^{k+1}_0\gets\max\bigcurl{0, \bar{\mu}^{k}_0+\rho^{k}g_{0,N^{k}}(x^{k})}, \\
&\mu^{k+1}_i\gets\max\bigcurl{0, \bar{\mu}^{k}_i+\rho^{k}g_i(x^{k})}\quad\forall i\in\I, \\
&\bar{\mu}^{k+1}_i\gets\min\{\mumax,\mu^{k+1}_i\}\quad\forall i\in\I_0.
\ea
\eeq

\If{$\sigma_{N^k}(x^k,\mu^{k+1})>\eta^k$}
{$\rho^{k+1}\gets\theta_{\rho}\rho^{k}$ \label{lin:rho-update}} 
\Else{$\rho^{k+1}\gets\rho^{k}$}
}
\caption{\fontsize{8}{8}\selectfont A probabilistic augmented Lagrangian method for solving \eqref{opt:NLP}\label{alg:penalty}}
\end{algorithm2e}

\section{Convergence Analysis}
\label{sec:convg-analy}
We analyze the convergence of \Cref{alg:penalty}
when the quantile function $Q^{1-\alpha}(\cdot)$ is differentiable. 
We first analyze the probabilistic properties of local model used
in the trust-region method for approximating the merit function (\Cref{sec:loc-model-approx}). 
Then use these results to analyze the convergence of \Cref{alg:trust-region-merit-func}
for solving the merit-function minimization problem (\Cref{sec:convg-trust-region}).
We provide a global convergence result of \Cref{alg:penalty} (\Cref{sec:global-convg}).
The parameter updating rule in this algorithm is similar to \cite{birgin-practical-aug-lag,birgin2012-bd-penalty-param-aug-lag-ineq-constr},
which is a well recognized augmented Lagrangian method for solving deterministic nonlinear programs with equality and inequality constraints.
Note that the convergence of \Cref{alg:penalty} also relies on the
parameter setting, in particular, the sample size drawn from $\xi$
at every main iteration. The conditions on the sample size are 
specified in theorems of convergence.

\subsection{Probabilistic properties of local model approximation}
\label{sec:loc-model-approx}
In each iteration of \Cref{alg:trust-region-merit-func}, a quadratic model is
constructed as a local approximation of the merit function. Putative iterates
are produced by minimizing this model in a trust region. 
Clearly, the sufficient approximation accuracy is needed to ensure convergence of the
algorithm. A notion of $\ck$-fully linearity is introduced to characterize the approximation
accuracy within a neighborhood, which is formally given in \Cref{def:fully-linear-model}.
The probabilistic counterpart is given in \Cref{def:prob-fully-linear-model}
for a random local model. These definitions are brought from~\cite{larson2014-deriv-free-stoch-opt}.
\begin{definition}\label{def:fully-linear-model}
Let $f$ be continuously differentiable,
let $\ck \defined (\ckeg,\ckef)$ be a given vector of absolute
constants, and let $\Delta>0$ be given. 
A function $m_f\in C^1$ is a $\ck$-\textit{fully linear}
model of $f$ on $B(x,\Delta)$ if $\nabla m_f$ is Lipschitz continuous 
and for all $x\in\dom$ and $y\in B(x,\Delta)$, 
\beq
\norm{\nabla f(y)-\nabla m_f(y)} \le \ckeg\Delta,\textrm{ and } 
|f(y)-m_f(y)|\le \ckef\Delta^2,
\eeq
where $\kappa$ and $\Delta$ are independent of $x$.
\end{definition}

\begin{definition}\label{def:prob-fully-linear-model}
Consider running an algorithm to generate an infinite sequence of points $\{x^k\}^\infty_{k=1}$,
and let $\mc{F}_k$ be the $\sigma$-algebra representing the information available at iteration $k$.
Let $\ck \defined (\ckef,\ckeg)$ be a given vector of constants, let $\varepsilon\in(0,1)$, and let $\Delta>0$
be given. A random model $m_f$ generated based on samples of random parameters is 
$(1-\varepsilon)$-probabilistically 
$\ck$-fully linear on $B(x,\Delta)$ if 
\beq
\P(m_f\textrm{ is a $\ck$-fully linear model of $\Phi$ on $B(x,\Delta)$}|\mc{F}_{k-1})\ge 1-\varepsilon,
\eeq
where $\kappa$ and $\Delta$ are independent of $x$.
\end{definition}

\begin{definition}\label{def:local-approx-acc}
For fixed parameters $\rho,\mu$,  
the quadratic models in the trust-region \Cref{alg:trust-region-merit-func}
satisfy the ($\varepsilon$, $\theta$)-\textit{probabilistic local approximation accuracy} 
if there exists a $\ol{k}$ such that for any iteration $k>\ol{k}$
the following two conditions are satisfied:
\beq\label{eqn:local-approx-acc}
\ba
&\P\left[ \left|V_k\right|>
		\eta_1\eta_2\Delta^2_k \Big| \mc{F}_{k-1} \right]\le \varepsilon, \\
&\P\left[ \left|V_k\right|>
		(\eta_1\eta_2+w)\Delta^2_k \Big| \mc{F}_{k-1} \right]\le \frac{\theta}{w} \quad\forall w>0, \\
&\text{with } V_k \defined \Phi_{N_k}(x^k)-\Phi(x^k)+\Phi(x^k+s^k)-\Phi_{N_k}(x^k+s^k),
\ea
\eeq
where the Lagrangian and penalty parameters in the functions $\Phi_{N_k}$ and $\Phi$ are omitted here for simplicity,
$\Delta_k\le 1$ is the trust-region radius at iteration $k$, 
$s^k$ is trust-region subproblem solution at iteration $k$,
$\mc{F}_k$ is the $\sigma$-algebra representing 
the information available at iteration $k$,
and $\eta_1,\eta_2$ are trust-region updating parameters in \Cref{alg:trust-region-merit-func}.
\end{definition}

\begin{proposition}\label{prop:g(x)-GN(x0)}
Suppose \Cref{ass:f-c2-diff,ass:c1-diff,ass:Q-diff,ass:F_twice_diff} hold.
Let $L_g$ be a shared Lipschitz constant of $\nabla g_i$ for all $i\in\I_0$.
For any constants $a\ge0$, $\gamma\in(0,1)$, and any $x_0,x\in\dom$ satisfying $\norm{x-x_0}\le\Delta\le1$, 
suppose $\beta=r_0\Delta$ with $r_0<\frac{2L_g}{nL_Q}$ and 
$N\ge O\left( \frac{C^2_\dom\alpha(1-\alpha)n^2(\log\frac{n}{\gamma})^3}{d^2_\dom L^2_gr^2_0\Delta^2} \right)$,
then the following three inequalities hold:
\[
\ba
&\P\parenth{\norm{\nabla g_0(x) - G_N(x_0)}\ge 2L_g\Delta }\le\gamma,  \\
&\P\Big\{\norm{\max\bigcurl{0, g_0(x)+a}\nabla{g_0(x)} - \max\bigcurl{0, g_{0,N}(x_0)+a} G_N(x_0)}  \\
&\quad\ge\big[6\parenth{ \norm{\nabla g_0(x_0)}+|g_0(x_0)|+L_g}^2+3aL_g\big]\Delta\Big\}\le\gamma, \\
&\norm{\max\{0,g_i(x)+a\}\nabla g_i(x) - \max\{0,g_i(x_0)+a\}\nabla g_i(x_0) }  \\
&\le\big[2(|g_i(x)| + \norm{\nabla g_i(x_0)} + L_g)^2 + aL_g\big]\Delta. 
\ea
\]
\end{proposition}
\begin{proof}
To prove the first inequality, we notice that
\bdm
\ba
\norm{\nabla g_0(x) - G_N(x_0)} &\le \norm{\nabla g_0(x)-\nabla g_0(x_0)} + \norm{\nabla g_0(x_0) - G_N(x_0)} \\
&\le L_g\norm{x-x_0} + \norm{\nabla Q^{1-\alpha}(x_0) - \wh{D}_N(x_0)  },
\ea
\edm
where $\wh{D}_N(\cdot)$ is defined in \eqref{eqn:whG}.
Because 
$N\ge O\left( \frac{C^2_\dom\alpha(1-\alpha)n^2(\log\frac{n}{\gamma})^3}{d^2_\dom L^2_gr^2_0\Delta^4} \right)$, 
we can apply \Cref{thm:quant-grad-approx} to conclude that
\bdm
\ba
&\P\parenth{\norm{\nabla g_0(x) - G_N(x_0)}\ge L_g\Delta}\le \gamma,
&\P\parenth{\norm{\nabla g_0(x) - G_N(x_0)}\ge 2L_g\Delta}\le \gamma,
\ea
\edm
where the condition $r_0<\frac{2L_g}{nL_Q}$ ensures that $\beta<\frac{2L_g\Delta}{nL_Q}$ 
(a condition required by \Cref{thm:quant-grad-approx}). This proves the first inequality
of the proposition.
The second inequality of the proposition can be bounded as follows

\beq\label{eqn:|maxg-maxG|}
\ba
&\norm{\max\bigcurl{0, g_0(x)+a}\nabla{g_0(x)} - \max\bigcurl{0, g_{0,N}(x_0)+a} G_N(x_0)} \\
&\le\norm{\max\bigcurl{0, g_0(x)+a}\nabla{g_0(x)} - \max\bigcurl{0, g_0(x)+a}\nabla{g_0(x_0)} } \\
&\quad+\norm{\max\bigcurl{0, g_0(x)+a}\nabla{g_0(x_0)} - \max\bigcurl{0, g_0(x_0)+a}\nabla{g_0(x_0)} } \\
&\quad+\norm{\max\bigcurl{0, g_0(x_0)+a}\nabla{g_0(x_0)} - \max\bigcurl{0, g_0(x_0)+a}G_N(x_0) } \\
&\quad+\norm{\max\bigcurl{0, g_0(x_0)+a}G_N(x_0) - \max\bigcurl{0, g_{0,N}(x_0)+a}G_N(x_0) } \\
&\le\left(|g_0(x)|+a\right)\norm{\nabla{g_0}(x)-\nabla{g_0}(x_0)} + |g_0(x)-g_0(x_0)|\cdot\norm{\nabla{g_0}(x_0)} \\
&\quad+\left(|g_0(x_0)|+a\right)\norm{\nabla{g_0}(x_0)-G_N(x_0)} + |g_0(x_0)-g_{0,N}(x_0)|\cdot\norm{G_N(x_0)},
\ea
\eeq
where we use the property $|\max\{0,a\}-\max\{0,b\}|\le|a-b|$.
We know that with probability at least $1-2\gamma$ the following inequalities hold jointly
\beq\label{eqn:interm1}
\ba
&\norm{\nabla{g_0(x)}-G_N(x_0)}\le 2L_g\Delta, \\
&\norm{G_N(x_0)}\le\norm{\nabla{g_0}(x_0)} + \norm{\nabla{g_0}(x_0) - G_N(x_0)}\le\norm{\nabla{g_0}(x_0)}+2L_g\Delta, \\
&|g_0(x_0)-g_{0,N}(x_0)|\le{L_g}\Delta\; \textrm{(\Cref{cor:P(|Q-QN|)})}.
\ea
\eeq
Substituting \eqref{eqn:interm1} into \eqref{eqn:|maxg-maxG|}, we conclude that
with probability at least $1-2\gamma$, the following inequality holds:

\beq\label{eqn:|maxg-maxG|_2}
\ba
&\norm{\max\bigcurl{0, g_0(x)+a}\nabla{g_0(x)} - \max\bigcurl{0, g_{0,N}(x_0)+a} G_N(x_0)} \\
&\le\left(|g_0(x)|+a\right)L_g\Delta + \left(\norm{\nabla g_0(x_0)}\Delta+\frac{1}{2}L_g\Delta^2\right)\norm{\nabla g_0(x_0)} \\
&\quad+2\left(|g_0(x_0)|+a\right)L_g\Delta + \norm{G_N(x_0)}L_g\Delta \\
&\le\left(|g_0(x)|+a\right)L_g\Delta + \left(\norm{\nabla g_0(x_0)}\Delta+\frac{1}{2}L_g\Delta^2\right)\norm{\nabla g_0(x_0)} \\
&\quad+2\left(|g_0(x_0)|+a\right)L_g\Delta + (\norm{\nabla{g_0}(x_0)}+2L_g\Delta)L_g\Delta \\
&\le \left(|g_0(x)|+a\right)L_g\Delta + 2\left(|g_0(x_0)|+a\right)L_g\Delta  \\
&\quad+ \left(\norm{\nabla g_0(x_0)} + \frac{1}{2}L_g\right)\norm{\nabla g_0(x_0)}\Delta  + (\norm{\nabla{g_0}(x_0)}+2L_g)L_g\Delta.
\ea
\eeq
Applying the inequality 
\bdm
\ba
|g_0(x)|&\le|g_0(x_0)|+\norm{\nabla g_0(x_0)}\Delta+L_g\Delta^2/2 \\
&\le|g_0(x_0)|+\norm{\nabla g_0(x_0)}+L_g/2
\ea
\edm
to bound $|g_0(x)|$ in the first term of \eqref{eqn:|maxg-maxG|_2} and enlarging the coefficients leads to
the following inequality with probability at least $1-2\gamma$:

\beq
\ba
&\norm{\max\bigcurl{0, g_0(x)+a}\nabla{g_0(x)} - \max\bigcurl{0, g_{0,N}(x_0)+a} G_N(x_0)} \\
&\le3(|g_0(x_0)|+\norm{\nabla g_0(x_0)}+L_g+a)L_g\Delta \\
&\quad + \left(\norm{\nabla g_0(x_0)} + L_g\right)\norm{\nabla g_0(x_0)}\Delta  + (\norm{\nabla{g_0}(x_0)}+2L_g)L_g\Delta \\
&\le[6(|g_0(x_0)|+\norm{\nabla g_0(x_0)}+L_g)^2 + 3aL_g]\Delta.
\ea
\eeq
This proves the second inequality of the proposition.
Note that some constant factors 
have been absorbed in the sample complexity, for example
the factor 2 in $1-2\gamma$. The third 
inequality can be proved similarly as the second inequality.
\qed
\end{proof}

\begin{proposition}\label{prop:gamma-prob-linear-model}
Suppose \Cref{ass:f-c2-diff,ass:c1-diff,ass:Q-diff,ass:F_twice_diff} hold.
Let $x_0\in\dom$ be any point and $x\in{B}(x_0, \Delta)$. 
Let $0 < \Delta < 1$ be a trust-region radius and $r_0>0$ be constant. 
Set $\beta \defined r_0\Delta$ with $r_0<\frac{2L_g}{nL_Q}$.
If $N\ge O\left( \frac{C^2_\dom\alpha(1-\alpha)n^2(\log\frac{n}{\gamma})^3}{d^2_\dom L^2_gr^2_0\Delta^4} \right)$,
then $m_N(x,\rho,\mu,\beta)$ is a $(1-\gamma)$-probabilistic
$\ck$-fully linear model (\Cref{def:prob-fully-linear-model}) on $B(x_0,\Delta)$
with the parameter $\ck \defined (\ckeg,\ckef)$ 
where $\ckeg \defined 8K_{\dom}|\I_0|\left(1+\sum_{i\in\I_0}\mu_i+\rho\right)$,  
$\ckef \defined 6K_{\dom}|\I_0|\left(1+\sum_{i\in\I_0}\mu_i+\rho\right)$, and $K_{\dom}$
is a constant satisfying 
\begin{equation}\label{def:C}
\begin{aligned}
&K_\dom \ge L_f+L_g+\norm{H} ,\\
&K_\dom \ge \max_{x\in\ol{\dom}} \left(\norm{\nabla{g_i(x)}}+\abs{g_{i}(x)} +L_g\right)^2 \;\forall i\in\mc{I}_0.
\end{aligned}
\end{equation}
\end{proposition}
\begin{proof}
We omit the argument $\rho,\mu,\beta$ in the functions 
$\Phi(x,\rho,\mu)$, $\phi(x,\rho,\mu,\beta)$ 
and $m_N(x,\rho,\mu,\beta)$ to simplify notations in the proof.
First, applying the model definitions \eqref{eqn:grad-Phi}, \eqref{eqn:phi_N},
and \eqref{eqn:m_N} yields:
\beq\label{eqn:grad_Phi-m}
\ba
&\quad\norm{\nabla\Phi(x)-\nabla{m_N(x)}} \\
&=\norm{ \nabla{f(x)}+\rho\sum_{i\in\I_0}\max\bigcurl{0, g_i(x)+\frac{\mu_i}{\rho}}\nabla{g_i(x)} 
-\phi_N(x_0)-H(x-x_0)} \\
&\le\norm{\nabla{f(x)}-\nabla{f(x_0)}} + \norm{H(x-x_0)} \\
&\quad+\rho\norm{\max\bigcurl{0, g_0(x)+\frac{\mu_0}{\rho}}\nabla{g_0(x)} -\max\bigcurl{0, g_{0,N}(x_0)+\frac{\mu_0}{\rho}} G_N(x_0)} \\
&\quad+\sum_{i\in\I}\rho\norm{\max\bigcurl{0, g_i(x)+\frac{\mu_i}{\rho}}\nabla{g_i(x)} -\max\bigcurl{0, g_i(x_0)+\frac{\mu_i}{\rho}} \nabla g_i(x_0)}.
\ea
\eeq
By applying the second and the third inequalities from Proposition~\ref{prop:g(x)-GN(x0)},
we conclude that with probability at least $1-2\gamma$ the following inequalities hold:
\beq\label{eqn:grad_Phi-mN_final}
\ba
&\quad\norm{\nabla\Phi(x)-\nabla{m_N(x)}} \\
&\le(\norm{H}+L_f)\Delta +  \big[6\rho\parenth{ \norm{\nabla g_0(x_0)}+|g_0(x_0)|+L_g}^2+3\mu_0L_g\big]\Delta \\
&\quad+\sum_{i\in\I}\big[2\rho(|g_i(x)| + \norm{\nabla g_i(x_0)} + L_g)^2 + \mu_iL_g\big]\Delta \\
&\le{K_\dom}\Delta + (6\rho{K_\dom} + 3\mu_0{K_\dom})\Delta + \sum_{i\in\I}(2\rho{K_\dom}+\mu_i{K_\dom})\Delta \\
&\le8|\I_0|\left(1+\rho+\sum_{i\in\I_0}\mu_i\right){K_\dom}\Delta:=\ckeg\Delta.
\ea
\eeq
Using integral representation, we can rewrite $\Phi(x)-m_N(x)$ as:
\bdm
\ba
&\Phi(x)-m_N(x)=\Phi(x_0)-m_N(x_0) \\
&\quad +\int^1_0[\nabla\Phi((1-t)x_0+tx)^\top-\nabla m_N((1-t)x_0+tx)^\top](x-x_0)dt.
\ea
\edm
Then we can obtain the following inequality to bound $\abs{\Phi(x)-m_N(x)}$:
\bdm
\ba
&\abs{\Phi(x)-m_N(x)}\le \abs{\Phi(x_0)-m_N(x_0)} \\
&+ \int^1_0\norm{\nabla\Phi((1-t)x_0+tx)-\nabla m_N((1-t)x_0+tx)}\norm{x-x_0}dt.
\ea
\edm
The two terms on the right side of the above inequality can be bounded as follows
with joint probability at least $1-2\gamma$:

\beq\label{eqn:|Phi(x0)-mN(x0)|}
\ba
&\abs{\Phi(x_0)-m_N(x_0)}=\abs{\Phi(x_0)-\Phi_N(x_0)} \\
&=\abs{\frac{\rho}{2}\max\left\{0, g_0(x_0)+\frac{\mu_0}{\rho}\right\}^2 - \frac{\rho}{2}\max\left\{0, g_{0,N}(x_0)+\frac{\mu_0}{\rho}\right\}^2 } \\
&\le\frac{\rho}{2}\abs{\max\left\{0, g_0(x_0)+\frac{\mu_0}{\rho}\right\} + \max\left\{0, g_{0,N}(x_0)+\frac{\mu_0}{\rho}\right\} }
\cdot |g_0(x_0) - g_{0,N}(x_0)| \\
&\le\frac{\rho}{2}\abs{2\max\left\{0, g_0(x_0)+\frac{\mu_0}{\rho}\right\}+ |g_0(x_0) - g_{0,N}(x_0)|}\cdot |g_0(x_0) - g_{0,N}(x_0)| \\
&\le\left(\rho|g_0(x_0)| + \mu_0 +\frac{\rho}{2}|g_0(x_0) - g_{0,N}(x_0)| \right)\cdot |g_0(x_0) - g_{0,N}(x_0)| \\
&\le\left(\rho|g_0(x_0)| + \mu_0 +\frac{\rho}{2}L_g\Delta^2 \right)\cdot L_g\Delta^2 \;\text{(\Cref{cor:P(|Q-QN|)} and condition on $N$)} \\
&\le\left(\rho|g_0(x_0)| + \mu_0 +\frac{\rho}{2}L_g \right)\cdot L_g\Delta^2, 
\ea
\eeq
\beq
\int^1_0\norm{\nabla\Phi((1-t)x_0+tx)-\nabla m_N((1-t)x_0+tx)}\norm{x-x_0}tdt\le \frac{1}{2}\ckeg\Delta^2,
\eeq
where \eqref{eqn:grad_Phi-mN_final} is used to get the second inequality.
Therefore, we have that with probability at least $1-2\gamma$ the following inequality holds
\beq\label{eqn:Phi-mN}
\ba
\abs{\Phi(x)-m_N(x)}&\le\left(\rho|g_0(x_0)| + \mu_0 +\frac{\rho}{2}L_g \right)\cdot L_g\Delta^2 + \frac{1}{2}\ckeg\Delta^2 \\
&\le \frac{1}{2}\ckeg\Delta^2 + 2\Kdom\left(1+\sum_{i\in\I_0}\mu_i+\rho\right)\Delta^2 \\
&\le 6|\I_0|\left(1+\sum_{i\in\I_0}\mu_i+\rho\right)\Kdom\Delta^2:=\ckef\Delta^2.
\ea
\eeq
Inequalities \eqref{eqn:grad_Phi-mN_final} and \eqref{eqn:Phi-mN} conclude the proof
by absorbing the constant factor 2 associated with $\gamma$ into the sample complexity.
\qed
\end{proof}
\begin{remark}
The $N\sim 1/\Delta^4$ relationship between sample size and the trust-region size
matches that in~\cite{larson2014-deriv-free-stoch-opt}.
\end{remark}
\begin{theorem}\label{thm:(e,theta)-prob-loc-approx}
Suppose \Cref{ass:f-c2-diff,ass:c1-diff,ass:Q-diff,ass:F_twice_diff} hold.
Consider \Cref{alg:trust-region-merit-func} for given parameters $\rho,\mu$
and internal parameters $\derate,\inrate$, $\eta_1,\eta_2,r_0,\Delta_0$.
The following properties hold jointly \newline
$\mathrm{(a)}$ $m_{N_k}$ from \eqref{eqn:m_N} is a $(1-\varepsilon)$-probabilistically $\ck$-fully linear model, 
where $\ck \defined (\ckeg,\ckef)$, $\ckeg \defined 8\Kdom|\I_0|\parenth{1+\sum_{i\in\I_0}\mu_i+\rho}$, 
$\ckef \defined 6\Kdom|\I_0|\parenth{1+\sum_{i\in\I_0}\mu_i+\rho}$;
\newline
$\mathrm{(b)}$ \Cref{alg:trust-region-merit-func} satisfies the 
($\varepsilon$, $\theta$)-probabilistic local approximation
accuracy condition (\Cref{def:local-approx-acc})
provided that the number of samples $N_k$ drawn 
from the distribution of the random parameters $\xi$ satisfies
\beq\label{eqn:N_cond_kappa_alpha}
N_k\ge O\left(\frac{\Cdom^2n^2\alpha(1-\alpha)\left(\log\frac{n}{\varepsilon}\right)^3}{A_k\Delta^4_k}\right),
\eeq
assuming $\Delta_k<1$, where the factor $A_k$ is defined as
\beq\label{eqn:Ak}
\begin{aligned}
&A_k \defined \min\parenthcurl{
d^2_\dom L^2_gr^2_0,\; d^2_\dom\rho\eta_1\eta_2,\;d^2_\dom\rho\sqrt{\eta_1\eta_2\theta}\;
\frac{d^2_\dom\eta^2_1\eta^2_2}{(\mu_0+{\rho}K_\dom)^2},\;
\frac{d^2_\dom\eta_1\eta_2\theta}{(\mu_0+{\rho}K_\dom)^2}}.
\end{aligned}
\eeq
\end{theorem}
\begin{proof}
To simplify notation in the proof, we omit the iteration index $k$ in \Cref{alg:trust-region-merit-func} and
the arguments $\rho,\mu,\beta$ in the functions $\Phi(x,\rho,\mu)$ and $\Phi_N(x,\rho,\mu)$.
\newline
(a) 
 From \Cref{prop:gamma-prob-linear-model}, if
 $N\ge O\left( \frac{C^2_\dom\alpha(1-\alpha)n^2(\log\frac{n}{\varepsilon})^3}{d^2_\dom L^2_gr^2_0\Delta^4} \right)$,
$m_N$ is a $(1-\varepsilon)$-probabilistically $\ck$-fully linear model,
with $\ckeg \defined 8\Kdom|\I_0|\parenth{1+\sum_{i\in\I_0}\mu_i+\rho}$ \newline
and $\ckef \defined 6\Kdom|\I_0|\parenth{1+\sum_{i\in\I_0}\mu_i+\rho}$. 
This sample complexity is dominated by \eqref{eqn:N_cond_kappa_alpha}.
\newline
(b)
It suffices to show that the model $m_{N_k}$ in \Cref{alg:trust-region-merit-func} 
satisfies the ($\varepsilon$, $\theta$)-probabilistic local approximation accuracy condition.
Note that the probability in the first condition of \eqref{eqn:local-approx-acc}
can be estimated using the union bound:
\[
\small
\ba
&\P\left[ \left|\Phi_N(x)-\Phi(x)+\Phi(x+s)-\Phi_N(x+s)\right|>
		\eta_1\eta_2\Delta^2\Big| \mc{F}_{k-1} \right] \\
&\le\P\left[ \left| \Phi_N(x)-\Phi(x) \right|>\frac{1}{2} \eta_1\eta_2\Delta^2 \Big|\mc{F}_{k-1} \right] +\P\left[ \left| \Phi_N(x+s)-\Phi(x+s) \right|>\frac{1}{2} \eta_1\eta_2\Delta^2 \Big|\mc{F}_{k-1} \right],
\ea
\]
where $\Phi_N$ is defined in \eqref{eqn:Phi_N}.
Similar as the first a few steps in the derivation of \eqref{eqn:|Phi(x0)-mN(x0)|}, the term $\abs{\Phi_N(x)-\Phi(x)}$ can be bounded deterministically as:
\bdm
\ba
&\abs{\Phi_N(x)-\Phi(x)}\le\abs{\frac{\rho}{2}\max\left\{0, g_{0,N}(x)+\frac{\mu_0}{\rho}\right\}^2 - \frac{\rho}{2}\max\left\{0, g_0(x)+\frac{\mu_0}{\rho}\right\}^2 } \\
&\le \left(\rho|g_0(x)| + \mu_0 +\frac{\rho}{2}|g_0(x) - g_{0,N}(x)| \right)\cdot |g_0(x) - g_{0,N}(x)|.
\ea
\edm
\Cref{cor:P(|Q-QN|)} implies that if
$N\ge O\left( \frac{C^2_\dom\alpha(1-\alpha)(\log\frac{1}{\varepsilon})^3}{d^2_\dom t^2\Delta^4} \right)$, \newline
then
$
\P\parenth{\abs{g_{0,N}(x)-g_0(x)}\le t \Delta^2 }\ge 1-\frac{\varepsilon}{2}.
$
(The constant $1/2$ in the probability can be absorbed into the complexity.)
Then with probability at least $1-\frac{\varepsilon}{2}$ 
\bdm
\ba
&\abs{\Phi_N(x)-\Phi(x)}\le
\parenth{\mu_0+\rho |g_0(x)| +\frac{\rho}{2}t\Delta^2 } t\Delta^2
\le\parenth{\mu_0+\rho{K_\dom} +\frac{\rho}{2}t\Delta^2 } t\Delta^2,
\ea
\edm
where $t$ is a parameter to be determined.
We can set $t$ to satisfy $(\mu_0+\rho{K_\dom})t\le\frac{1}{4}\eta_1\eta_2$
and $\frac{\rho}{2}t^2\Delta^2\le\frac{1}{4}\eta_1\eta_2$ to ensure that
\bdm
\P\parenth{\abs{\Phi_N(x)-\Phi(x)}\le\frac{1}{2}\eta_1\eta_2\Delta^2 \Big| \mc{F}_{k-1}}\ge 1-\frac{\varepsilon}{2}.
\edm
This means if $N$ satisfies
\beq\label{eqn:N>log1/epsilon_D^4}
N\ge 
O\left(
\frac{C^2_\dom \alpha(1-\alpha)(\log\frac{1}{\varepsilon})^3}{
\min\left\{\frac{d^2_\dom\eta^2_1\eta^2_2\Delta^4}{(\mu_0+\rho{K_\dom})^2},\;d^2_\dom\rho\eta_1\eta_2\Delta^2\right\}} 
\right),
\eeq 
then 
\beq\label{eqn:P(PhiN-Phi)-interm1}
\ba
&\P\parenth{\abs{\Phi_N(x)-\Phi(x)}\ge \frac{1}{2}\eta_1\eta_2\Delta^2 \Big| \mc{F}_{k-1}}\le \frac{\varepsilon}{2}, \textrm{ and hence } \\
&\P\left[ \left|\Phi_N(x)-\Phi(x)+\Phi(x+s)-\Phi_N(x+s)\right|\ge
		\eta_1\eta_2\Delta^2 \Big| \mc{F}_{k-1} \right]\le \varepsilon,
\ea
\eeq
i.e., the first condition of \eqref{eqn:local-approx-acc} holds, where we have used the fact that
the complexity condition \eqref{eqn:N>log1/epsilon_D^4} is independent of the position and hence
it can be applied to the position $x$ and $x+s$.
Next, we show that the second condition of \eqref{eqn:local-approx-acc} holds. Similarly, we have
\[
\ba
&\P\parenthcurl{\abs{\Phi_N(x)-\Phi(x)+\Phi(x+s)-\Phi_N(x+s)}>(\etaprod+w)\Delta^2\Big|\mc{F}_{k-1}  } \\
&\le\P\parenthcurl{\abs{\Phi_N(x)-\Phi(x)}>\frac{1}{2}(\etaprod+w)\Delta^2\Big|\mc{F}_{k-1}  } \\
&\quad+\P\parenthcurl{\abs{\Phi_N(x+s)-\Phi(x+s)}>\frac{1}{2}(\etaprod+w)\Delta^2\Big|\mc{F}_{k-1}  }.
\ea
\]
It suffices to find the condition for $N$ such that 
\beq\label{eqn:P(|PhiN-Phi|)-interm2}
\P\parenthcurl{\abs{\Phi_N(x)-\Phi(x)}\ge\frac{1}{2}(\etaprod+w)\Delta^2\Big|\mc{F}_{k-1}}\le\frac{\theta}{2w}.
\eeq
Similarly as \eqref{eqn:P(PhiN-Phi)-interm1}, a sufficient condition of $N$
for \eqref{eqn:P(|PhiN-Phi|)-interm2} to hold is:
\beq\label{eqn:N_complexity_min}
\ba
N\ge O\left(\frac{\Cdom^2\alpha(1-\alpha)(\log\frac{2w}{\theta})^3}{
\min\left\{\frac{d^2_\dom(\eta_1\eta_2+w)^2}{(\mu_0+\rho{K_\dom})^2}\Delta^4,\;
d^2_\dom \rho(\eta_1\eta_2+w)\Delta^2
\right\}
}\right)  \quad\forall w\ge \theta/2.
\ea
\eeq
The above complexity bound can be strengthened by applying the inequalities 
$\eta_1\eta_2+w\ge2\sqrt{\eta_1\eta_2w}$ to lower bound the denominator in 
\eqref{eqn:N_complexity_min} as (up-to a constant factor)
\beq
\min\left\{\frac{d^2_\dom\eta_1\eta_2w}{(\mu_0+\rho{K_\dom})^2}\Delta^4,\;
d^2_\dom \mu\sqrt{\eta_1\eta_2w}\Delta^2 
\right\}.
\eeq
It remains to find an upper bound for $\frac{(\log\frac{2w}{\theta})^3}{w}$
and $\frac{(\log\frac{2w}{\theta})^3}{\sqrt{w}}$ with $w>\theta/2$.
One can verify that the above two functions of $w$ achieve the maximum
value at $w=\frac{\theta e^3}{2}$ and $w=\frac{\theta e^6}{2}$, respectively.
This indicates that in order to ensure the second condition in \eqref{eqn:local-approx-acc} 
holds we can set $N$ to be
\beq\label{eqn:N-cond2}
N\ge O\left(\frac{C^2_\dom \alpha(1-\alpha)}{\min\left\{\frac{d^2_\dom \eta_1\eta_2\theta\Delta^4}{(\mu_0+\rho{K_\dom})^2},\;
d^2_\dom\rho\sqrt{\eta_1\eta_2\theta}\Delta^2 \right\}} \right)
\eeq
Notice that the sample complexity condition \eqref{eqn:N_cond_kappa_alpha}
dominates \eqref{eqn:N>log1/epsilon_D^4} and \eqref{eqn:N-cond2}, 
which concludes the proof.
\qed
\end{proof}

\subsection{Probability guarantee of the trust-region method for minimizing the merit function}
\label{sec:convg-trust-region}
The work~\cite{larson2014-deriv-free-stoch-opt} analyzes a probabilistic derivative-free 
trust-region method for solving an stochastic unconstrained optimization problem. 
The convergence result from Theorem~1 in \cite{larson2014-deriv-free-stoch-opt}
can be applied to our merit function when the Lagrangian parameters are fixed.
Therefore, it serves as a base of convergence analysis for the quantile constrained problem 
concerned here. The main result presented in \cite{larson2014-deriv-free-stoch-opt} is 
summarized in Theorem~\ref{thm:trust-region-deriv-free-unconstr}.

\begin{algorithm2e}[t]
  \fontsize{8}{8}\selectfont
  \SetAlgoNlRelativeSize{-4}
 Pick $0<\derate<1<\inrate$, $0<\eta_1,\eta_2,r_0<1$, $0<\Delta_0<1$. 
 
 \For{$k=0,1,\ldots$}{
 Choose $\beta \gets r_0\Delta_k$ for each iteration.
 Build a $(1-\varepsilon)$-probabilistically $\ck$-fully linear model $m_k$ 
 on $B(x^k,\Delta_k)$.
 
 Compute $s^k \gets \underset{s:\|s\|\le\Delta_k}{\textrm{arg min}}\;\; m_k(x^k+s)$

  \eIf{$m_k(x^k)-m_k(x^k+s^k,\lambda,\mu,\beta)\ge\eta_1\min\{\Delta_k,\Delta^2_k\}$}{
    Make a stochastic estimation of $h(x^k)$ and $h(x^k+s^k)$, which are denoted as $F_k(x^k)$
    and $F_k(x^k+s^k)$, respectively. 
    Calculate $\rho_k \gets \frac{F_k(x^k)-F_k(x^k+s^k)}{m_k(x^k)-m_k(x^k+s^k)}$. 

    \eIf{$\rho_k\ge\eta_2$}
    	{
		$x^{k+1} \gets x^k+s^k$;  $\Delta_{k+1} \gets \min\{1, \inrate\Delta_k\}$
    	}
    	{
		$x^{k+1} \gets x^k$;  $\Delta_{k+1} \gets \derate\Delta_k$
    	}
   }
   {
   $x^{k+1} \gets x^k$; $\Delta_{k+1} \gets \derate\Delta_k$
  }
  $k\gets k+1$.
  }
  \caption{\fontsize{8}{8}\selectfont Trust-region algorithm used in \cite{larson2014-deriv-free-stoch-opt} to minimize $h$.\label{alg:trust-region-general-func}}
\end{algorithm2e}

\begin{theorem}[\cite{larson2014-deriv-free-stoch-opt}]
\label{thm:trust-region-deriv-free-unconstr}
Let $h$ be a general smooth function that has bounded level sets and $\nabla h$ 
is Lipschitz continuous with constant $L_h$. 
Suppose Algorithm~\ref{alg:trust-region-general-func}
is applied to minimize $h$ and the function values of $h$
can only be accessed via the stochastic estimation model $F_k$ at any iteration $k$ with 
Lipschitz constant bounded by $L_F$. 
Suppose for every iteration $k$ that is greater than some threshold, the following two conditions hold: 
\beq\label{eqn:tr-cond1}
\mb{P}[|F_k(x^k)-h(x^k)+h(x^k+s^k)-F_k(x^k+s^k)|>\eta_1\eta_2\Delta^2_k|\mc{F}_{k-1}]\le \varepsilon
\eeq
and there is a constant $\theta>0$ such that 
\beq\label{eqn:tr-cond2}
\mb{P}[|F_k(x^k)-h(x^k)+h(x^k+s^k)-F_k(x^k+s^k)|>(\eta_1\eta_2+w)\Delta^2_k|\mc{F}_{k-1}]\le\frac{\theta}{w} \quad \forall w>0.
\eeq
Then 
\beq
\sum^\infty_{k=1}\Delta^2_k<\infty\;\textrm{ and }\; \lim_{k\to\infty} \norm{\nabla h(x^k)} = 0,
\eeq
almost surely
and the quantity 
\beq
\Psi_k = \max\left\{\frac{\norm{\nabla h(x^k)}}{\Delta_k},\;\mc{L}_2 \right\}
\eeq
is a supermartingale under the natural filtration, where 
\bdm
\ba
\mc{L}_1&=\max\left\{2L_F+\ckeg,\frac{(3-\eta_2)\ckeg+4\ckef}{1-\eta_2},\;\frac{2L_h}{\inrate-1},\;\frac{L_h}{\inrate}\left(\frac{1-\derate}{\derate}-\frac{1-\inrate}{\inrate} \right)^{-1} \right\}, \\
\mc{L}_2&=\max\left\{\frac{\mc{L}_1}{\derate},\; L_h+\mc{L}_1 \right\}.
\ea
\edm
\end{theorem}

\begin{theorem}[convergence of the trust-region algorithm]\label{thm:merit-convg}
Suppose \Cref{ass:f-c2-diff,ass:c1-diff,ass:Q-diff,ass:F_twice_diff} hold.
Suppose \Cref{alg:trust-region-merit-func} is applied 
with constants $\inrate$, $\derate$, fixed parameters $\rho,\mu$,
and the sample size $N_k$ satisfies \eqref{eqn:N_cond_kappa_alpha} at each iteration.
Let $\{x^k\}^{\infty}_{k=1}$ be the sequence generated from this algorithm when the termination
criterion is disregarded,
and $\{\Delta_k\}^{\infty}_{k=1}$ be the corresponding trust-region radii.
Then 
\beq
\sum^{\infty}_{k=1}\Delta^2_k<\infty,\quad and \quad
\lim_{k\to\infty} \|\nabla\Phi(x^k,\rho,\mu)\|=0
\eeq
almost surely.
Furthermore, there exists a parameter $\mc{L}(\rho,\mu,\Kdom)$, which only
depends on $\rho$, $\mu$, $\Kdom$ (defined in \eqref{def:C}), 
all the Lipschitz constants $L_f$, $L_{c_2}$, $L_Q$, $\norm{H}$ 
(the Hessian used for the local approximation model \eqref{eqn:m_N})
and parameters $\eta_1$, $\eta_2$, $\inrate$, $\derate$ in \Cref{alg:trust-region-merit-func},
such that the sequence $\{\Psi_k\}^{\infty}_{k=1}$ is a supermartingale under the natural filtration,
where 
\beq\label{eqn:Psi_k}
\Psi_k(\rho,\mu,\Kdom) \defined \max\left\{\frac{\|\nabla\Phi(x^k,\rho,\mu)\|}{\Delta_k},\mc{L}(\rho,\mu,\Kdom)\right\}.
\eeq 
\end{theorem}
\begin{proof}
The proof is a straightforward application of \Cref{thm:trust-region-deriv-free-unconstr}
on the merit function $\Phi(\cdot,\rho,\mu)$ with fixed penalty and Lagrangian parameters $\rho$ and $\mu$.
By Theorem~\ref{thm:(e,theta)-prob-loc-approx}, 
the assumptions and sample size condition in the theorem guarantee that the local models 
in \Cref{alg:trust-region-merit-func} satisfy the
($\varepsilon$, $\theta$)-probabilistic local approximation accuracy (\Cref{def:local-approx-acc}),
which is required by \Cref{thm:trust-region-deriv-free-unconstr} as \eqref{eqn:tr-cond1} and
\eqref{eqn:tr-cond2}. 

Note that in this case,  $\ckeg \defined 8\Kdom|\I_0|\parenth{1+\sum_{i\in\I_0}\mu_i+\rho}$ and
\newline
$\ckef \defined 6\Kdom|\I_0|\parenth{1+\sum_{i\in\I_0}\mu_i+\rho}$ from 
Theorem~\ref{thm:(e,theta)-prob-loc-approx}. The Lipschitz constant $L_\Phi$ of $\Phi$
can be bounded as
\bdm
\ba
L_\Phi &\le \left(1+\sum_{i\in\I_0}\mu_i+\rho\right)\max\{L_f,\; L_g,\; L_{g^2} \} \\
&\le \left(1+\sum_{i\in\I_0}\mu_i+\rho\right)\Kdom,
\ea
\edm
and hence the supermartingale result follows with $\Psi$ defined in \eqref{eqn:Psi_k}. 
\qed
\end{proof}

\Cref{thm:merit-convg} ensures the norm of the gradient of $\Phi$
converges to zero almost surely as the number of iterations goes to infinity.
However, we cannot run \Cref{alg:trust-region-merit-func}
without termination since the penalty parameters need to be adjusted in the outer iterations
and \Cref{alg:trust-region-merit-func} will be called repeatedly.
When the termination radius $r$ used in \Cref{alg:trust-region-merit-func}
is imposed, the algorithm will stop in finitely many iterations almost surely.  
In this case, the following theorem characterizes the quality of the solution.

\begin{theorem}\label{thm:P(|grad-Phi|)}
Suppose \Cref{ass:f-c2-diff,ass:c1-diff,ass:Q-diff,ass:F_twice_diff} hold.
Suppose \Cref{alg:trust-region-merit-func} is applied 
with fixed Lagrangian parameters $\rho,\mu$ and a termination radius $r$.
If the sample size $N^k$ satisfies \eqref{eqn:N_cond_kappa_alpha} at every iteration,
then the algorithm will terminate in finitely many iterations almost surely.
Let $x^{\tau}$ be the returned solution of the algorithm. 
Then
\begin{equation}\label{eqn:P(|grad-Phi|)}
\P\left(\|\nabla\Phi(x^{\tau},\rho,\mu)\|\le\delta\right)\ge 1-\frac{r}{\delta}\Psi_0(\rho,\mu,K_\dom) \qquad\forall \delta>0,
\end{equation}
where $\Psi_0(\rho,\mu,K_\dom)$ is defined in \eqref{eqn:Psi_k}
with $k=0$, which depends on  $\rho$, $\mu_i\;i\in\I_0$, $K_\dom$,
all the Lipschitz constants and internal parameters of \Cref{alg:trust-region-merit-func}.
\end{theorem}
\begin{proof}
We prove by contradiction. Suppose the algorithm does not terminate.
Given that \Cref{ass:f-c2-diff,ass:c1-diff,ass:Q-diff,ass:F_twice_diff} hold,
by Theorem~\ref{thm:(e,theta)-prob-loc-approx}(b), \Cref{alg:trust-region-merit-func} 
satisfies the $(\varepsilon,\theta)$-probabilistic local approximation accuracy condition.
Then by \Cref{thm:merit-convg}, we have 
$\sum^{\infty}_{k=1}\Delta^2_k<\infty$ almost surely.
It follows that there exists a first iteration $\tau$ such that $\Delta_\tau<r$ 
almost surely, which meets the termination criterion. 
Therefore, the algorithm terminates in finitely many iteration almost surely.
Let $\Psi_k \defined \max\left\{\frac{\|\nabla\Phi(x^k,\rho,\mu)\|}{\Delta_k},
\mc{L}(\rho,\mu,K_\dom)\right\}$,
which is a super-martingale according to \Cref{thm:merit-convg}.
The super-martingale property implies that $\E[\Psi_k]\le\Psi_0$. 
Therefore, it follows that
\bdm
\E\left[\frac{\|\nabla\Phi(x^k,\rho,\mu)\|}{\Delta_k}\right]\le\Psi_0\quad\forall k.
\edm
The termination criterion and the definition of $\tau$ implies
\bdm
\Delta_k>r \quad\forall k<\tau,\quad\textrm{and}\quad
r\ge\Delta_{\tau}=\derate\Delta_{\tau-1}\ge \derate r.
\edm
Using the optional stopping theorem \cite{durrett-probability-theory} with respect to $\tau$, we have 
\bdm
\E\left[\frac{\|\nabla\Phi(x^{\tau},\rho,\mu)\|}{r}\right]
\le\E\left[\frac{\|\nabla\Phi(x^{\tau},\rho,\mu)\|}{\Delta_{\tau}}\right]
\le \Psi_0,
\edm
which implies that
\bdm
\E[\|\nabla\Phi(x^{\tau},\rho,\mu)\|]\le r\Psi_0.
\edm
Using the Markov inequality, we have
\bdm
\ba
&\P\left(\|\nabla\Phi(x^{\tau},\rho,\mu)\|\le\delta\right)=1-\P\left(\|\nabla\Phi(x^{\tau},\rho,\mu)\|>\delta\right) \\
&\ge 1-\frac{\E[\|\nabla\Phi(x^{\tau},\rho,\mu)\|]}{\delta}\ge 1-\frac{r}{\delta}\Psi_0,
\ea
\edm
which concludes the proof.
\qed
\end{proof}

\subsection{Almost surely convergence}
\label{sec:global-convg}
We now analyze the convergence of \Cref{alg:penalty}.
We first present a technical lemma 
that will be frequently used in the remainder of the manuscript.
\Cref{lem:as-converg1} summarizes~\cite[Lemma~3.4]{loeve1951-almost-sure-convg},
which is implied by the Borel–Cantelli lemma and the Markov inequality.
\begin{lemma}\label{lem:as-converg1}
Let $X$ be a random variable and $\{X_k\}^{\infty}_{k=1}$ be a sequence of random variables. 
If for some $r>0$, $\sum^\infty_{k=1} \E[|X_k-X|^r]<\infty$, then $X_k\to X$ a.s.
\end{lemma}

The following lemma will be used to prove the almost sure convergence of the 
merit-function gradient later in \Cref{lem:|gradPhi|-to-0}.
\begin{lemma}\label{lem:as-converg2}
Let $\{X_m\}^{\infty}_{m=1}$ be a sequence of random variables and let \linebreak[4]
$\mc{F}_m=\sigma(\{X_i\}^m_{i=1})$ be the natural
filtration. Let $\{\delta_m\}^{\infty}_{m=1}$ and 
$\{\varepsilon_m\}^{\infty}_{m=1}$ be sequences of positive real numbers
that converge to zero. If $\E[X^2_m|\mc{F}_{m-1}]$
is almost surely bounded uniformly by a constant and there exists $\epsilon >0$
such that $\delta_m\le O(1/m^{1+\epsilon})$,
$\varepsilon_m\le O(1/m^{2+\epsilon})$, 
and $\P(|X_{m+1}|>\delta_{m+1}|\mc{F}_m)\le\varepsilon_{m+1}$ a.s.~for every $m$,
then $X_m\to 0$ a.s.
\end{lemma}
\begin{proof}
 We first apply \Cref{lem:as-converg1} to $\{X_m\}^\infty_{m=1}$
with $r=1$, $X_0=0$. So it suffices to verify $\sum^\infty_{m=1}\E[|X_m|]<\infty$.
The value $\E[|X_m|]=\E[\E[|X_m||\mc{F}_{m-1}]]$ can be bounded:
\beq
\begin{aligned}
&\E[|X_m||\mc{F}_{m-1}] = \E[|X_m|\bs{1}_{\{|X_m|>\delta_m\}}|\mc{F}_{m-1}]
+\E[|X_m|\bs{1}_{\{|X_m|\le\delta_m\}}|\mc{F}_{m-1}] \\
&\le \sqrt{\E[|X_m|^2|\mc{F}_{m-1}]\cdot\P(|X_m|>\delta_m|\mc{F}_{m-1})} 
+ \delta_m\le O(\sqrt{\varepsilon_m}+\delta_m) \quad a.s..
\end{aligned}
\eeq
Since $\varepsilon_m\le O(1/m^{2+\epsilon})$ and $\delta_m\le O(1/m^{1+\epsilon})$,
we have $\E[|X_m||\mc{F}_{m-1}]\le O(\sum_m 1/m^{1+\epsilon/2}
+\sum_m 1/m^{1+\epsilon})<\infty$ and the result is shown.
\qed
\end{proof}

\begin{lemma}\label{lem:inf-joint-event-convg-zero}
Let $\{\mc{F}_m\}^\infty_{m=1}$ be a filtration and $\{E_m\}^\infty_{m=1}$ be a sequence
of events such that $E_k$ is $F_m$-measurable for all $k<m$. 
Suppose $\P(E_m|\mc{F}_m)\le\delta_m$ for all $m$. Then 
\bdm
\P(\cap^m_{i=k}E_i) \le \prod^m_{i=k}\delta_i.
\edm
\end{lemma}
\begin{proof}
Let $\bs{1}_{E_i}$ be the indicator random variable of $E_i$. It follows that
\bdm
\ba
\P(\cap^m_{i=k}E_i)&=\E\left[\prod^m_{i=k}\bs{1}_{E_i}\right] 
=\E\left[\E\left[\prod^m_{i=k}\bs{1}_{E_i}\Bigg\vert \mc{F}_{m} \right]\right] \\
&=\E\left[\E\left[\bs{1}_{E_m}\vert \mc{F}_{m} \right]\cdot \prod^{m-1}_{i=k}\bs{1}_{E_i} \right] \\
&=\E\left[\P(E_m\vert \mc{F}_{m})\cdot \prod^{m-1}_{i=k}\bs{1}_{E_i} \right] \\
&\le\delta_m\E\left[ \prod^{m-1}_{i=k}\bs{1}_{E_i} \right].
\ea
\edm
Using induction, we can finally get
\bdm
\P(\cap^m_{i=k}E_i) \le \prod^m_{i=k}\delta_i. \tag*{\qed}
\edm
\qed
\end{proof}

We can now show that the gradient of the merit function converges to zero.
\begin{lemma}\label{lem:|gradPhi|-to-0}
Suppose \Cref{ass:f-c2-diff,ass:c1-diff,ass:Q-diff,ass:F_twice_diff} hold. 
Suppose \Cref{alg:penalty} is applied to \eqref{opt:NLP} to generate the sequences 
$\{x^{m}\}^{\infty}_{m=1}$ and $\{\rho^{m},\mu^{m},\bar{\mu}^m\}^{\infty}_{m=1}$.
Suppose the sample size in \Cref{alg:trust-region-merit-func} satisfies 
\eqref{eqn:N_cond_kappa_alpha} in each iteration.
Suppose the termination trust-region radius $r^{m}$ in 
\Cref{alg:penalty} satisfies $r^{m} \le{o}\left(\frac{1}{\Psi^{m}}\right)$
for some $\sigma>0$, where $\Psi^{m}=\Psi_0(\rho^m,\bar{\mu}^m,K_\dom)$ is defined in \eqref{eqn:Psi_k} with $k=0$.
Then for any subsequence $\mc{M}$, almost surely there exists a subsequence  
$\mc{M}^\prime\subseteq\mc{M}$ such that 
\beq
\lim_{m\in\mc{M}^\prime}\|\nabla\Phi(x^{m},\rho^{m},\bar{\mu}^{m})\|=0.
\eeq
\end{lemma}
\begin{proof}
By \Cref{thm:P(|grad-Phi|)},
\bdm
\P\left(\norm{\nabla\Phi(x^{m},\rho^{m},\bar{\mu}^{m})}>
\delta^{m}\biggr\rvert \mc{F}_{m-1}\right)\le\frac{r^m}{\delta^m}\Psi^m
\edm
holds for all
$\delta^m>0$
for every iteration $m\in\mc{M}$. 
Since $r^{m} \le{o}\left(\frac{1}{\Psi^{m}}\right)$, let $r^m=\frac{a_m}{\Psi^m}$
for some $a_m\to0$.
Set $\delta^m=\sqrt{a_m}$ and substitute the expression of $r^m$ and $\delta^m$ into 
the above inequality, we obtain 
\beq\label{eqn:P<1/n^2+sigma}
\P\left(\norm{\nabla\Phi(x^{m},\rho^{m},\bar{\mu}^{m})}>
\sqrt{a_m}\bigg\rvert \mc{F}_{m-1}\right)\le\sqrt{a_m}.
\eeq
Let $E_{m}$ be the event:
\bdm
\norm{\nabla\Phi(x^{m},\rho^{m},\bar{\mu}^{m})}\le\sqrt{a_m}.
\edm
The probability inequality \eqref{eqn:P<1/n^2+sigma} together with 
\Cref{lem:inf-joint-event-convg-zero} further imply that almost surely 
there exist an infinite number of events from the sequence $\{E_m:m\in\mc{M}\}$ that happen. 
Let $\{E_m:m\in\mc{M}^\prime\}$ be such a subsequence of events from $\{E_m:m\in\mc{M}\}$
that happen. In this subsequence, we have
\bdm
\lim_{m\in\mc{M}^\prime}\norm{\nabla\Phi(x^{m},\rho^{m},\bar{\mu}^{m})}
=\lim_{m\in\mc{M}^\prime} \sqrt{a_m} = 0. \tag*{\qed}
\edm
\end{proof}

The following lemma provides conditions on the sample size $N\up{n}$ in each
iteration of \Cref{alg:penalty} to ensure the convergence of the quantile-constraint evaluation based on samples. 
\begin{lemma}\label{lem:N-for-gN-convg}
Suppose \Cref{ass:f-c2-diff,ass:c1-diff,ass:Q-diff,ass:F_twice_diff} hold. 
Let $\{x^{m}\}^{\infty}_{m=1}$  
be the points generated from \Cref{alg:penalty}.
Suppose the sample size in \Cref{alg:trust-region-merit-func} satisfies \eqref{eqn:N_cond_kappa_alpha} in each iteration.
If the sample size $N^m$ from Step~2 of \Cref{alg:penalty} satisfies the condition 
\beq\label{eqn:N-cond-as-converg}
N^m\ge O\left( \frac{C^2_\dom\alpha(1-\alpha)(\log\frac{1}{\gamma^{m}})^3}{d^2_\dom\cdot(\delta^{m})^2} \right),
\eeq
with $\delta^{m}\le O(1/m^{1+\sigma})$,
$\gamma^{m}\le O(1/m^{2+\sigma})$
holding for some $\sigma>0$,
then 
\beq
\lim_{m\to\infty}|g_{0,N^m}(x^m)-g_0(x^m)|=0\quad a.s.
\eeq
\end{lemma}
\begin{proof}
Define a sequence of random variables 
$$X_m \defined \abs{g_{0,N^{m}}(x^{m})-g_0(x^{m})}=\abs{\Q{N^{m}}{x^{m}}-Q^{1-\alpha}(x^{m})}.$$
We will apply \Cref{lem:as-converg2} to $\{X_m\}^\infty_{m=1}$
for the desired result. We first show that the quantity $\E[X^2_m|\mc{F}_{m-1}]$
is bounded. Note that $\E[X^2_m|\mc{F}_{m-1}]$ has the form 
\begin{displaymath}
\begin{aligned}
\E\big[|g_{0,N^{m}}(x^{m})-g_0(x^{m})|^2|\mc{F}_{m-1}\big]
=\E\big[|\Q{N^{m}}{x^{m}}-Q^{1-\alpha}(x^{m})|^2|\mc{F}_{m-1}\big].
\end{aligned}
\end{displaymath}
\Cref{cor:P(|Q-QN|)} (given \Cref{ass:c1-diff,ass:Q-diff,ass:F_twice_diff}) shows that if 
$N\ge O\left( \frac{C^2_\dom\alpha(1-\alpha)(\log\frac{1}{\gamma})^3}{d^2_\dom\delta^2} \right)$,
the following tail bound holds for all $x\in \dom$:
\bdm
\P\parenth{\abs{\Q{N}{x}-\Qx}^2\ge\delta^2}\le\gamma.
\edm
If we let $Y \defined \abs{\Q{N}{x}-\Qx}^2$, the above inequality 
is equivalent to $1-F_Y(\delta^2)\le\gamma$, where $F_Y$ represents the CDF 
of the random variable $Y$. 
The sample condition
$N\ge O\left( \frac{C^2_\dom\alpha(1-\alpha)(\log\frac{1}{\gamma})^3}{d^2_\dom\delta^2} \right)$ 
implies that $\gamma_0\le\exp(-C\delta^{2/3}N^{1/3})$ for some constant $C>0$.
It follows that for any $x\in\dom$:
\bdm
\ba
\E\big[|\Q{N}{x}-\Qx|^2\big]=\int^\infty_{0}[1-F_Y(t)]dt
\le\int^\infty_{0}\exp(-CN^{\frac{1}{3}}t^{\frac{1}{3}})dt<\infty.
\ea
\edm
The above inequality indicates that the expectation is uniformly bounded by a constant.
We conclude that $\E[X^2_m|\mc{F}_{m-1}]$ is bounded.
Then by \Cref{cor:P(|Q-QN|)}, because 
$N^{m}\ge O\left( \frac{C^2_\dom\alpha(1-\alpha)(\log\frac{1}{\gamma^{m}})^3}{d^2_\dom\cdot(\delta^{m})^2} \right)$,
it follows that 
\bdm
\P(|X_m|>\delta^{m} | \mc{F}_{m-1})\le\gamma^{m}.
\edm
Applying \Cref{lem:as-converg2} to $\{X_m\}^\infty_{m=1}$ concludes the proof.
\qed
\end{proof}

\Cref{lem:onE*converg} and \Cref{thm:convg} are probabilistic counterparts of  
Lemma~2.3 and Theorem~2.1 in \cite{birgin2012-bd-penalty-param-aug-lag-ineq-constr}, respectively.

\begin{lemma}\label{lem:onE*converg}
Suppose \Cref{ass:f-c2-diff,ass:c1-diff,ass:Q-diff,ass:F_twice_diff} hold and 
\Cref{alg:penalty} has been applied to solve \eqref{opt:NLP}.
Let $\{x^{m},\rho^{m},\mu^{m},\bar{\mu}^{m}\}^\infty_{m=1}$ be the (probabilistic) sequence 
generated from \Cref{alg:penalty}. Let $x^*$ be a point and let $E^*$ be the event
\beq
E^*=\{\omega\in\Omega: \{x^m(\omega)\}\textrm{ has a limit point }x^*\}.
\eeq
Let $J=\{i\in\I_0: g_i(x^*)<0\}$. 
Suppose the sample size in \Cref{alg:trust-region-merit-func} satisfies \eqref{eqn:N_cond_kappa_alpha} in each iteration,
the trust-region radius $r^{m}$ in Step 1 of  \Cref{alg:penalty} satisfies the condition in \Cref{lem:|gradPhi|-to-0}
and the number of samples $N^{m}$ in Step 2 of \Cref{alg:penalty} satisfies the condition in \Cref{lem:N-for-gN-convg} for all iterations.
Then the following properties hold:\newline
(1) Almost surely on $E^*$, there exists a subsequence $\mc{M}$ satisfying
\beq\label{eqn:lim_I0-J}
\lim_{m\in\mc{M}}\norm{\nabla{f(x^m)} + \sum_{i\in\I_0\setminus J}\mu^{m+1}_i\nabla{g_i(x^m)} }=0.
\eeq
(2) If $x^*$ is feasible and the MFCQ constraint qualification is satisfied by \eqref{opt:NLP} at $x^*$, 
then almost surely on $E^*$, the $\{\mu^{m+1}: m\in\mc{M}\}$ are bounded.
In this case, $x^*$ satisfies the KKT conditions, and if there is only one $\mu^*$
satisfying $\nabla{f(x^*)}+\sum_{i\in\I_0}\mu^*_i\nabla{g_i(x^*)}=0$ and $\mu^*{\ge}0$,
we have
\beq
\lim_{m\in\mc{M}}\mu^{m+1}=\mu^*\quad a.s.\textrm{ on }E^*.
\eeq
\end{lemma}

\begin{proof}

Part (1). By \Cref{lem:|gradPhi|-to-0}, almost surely on $E^*$ there exists a subsequence 
indexed by $\mc{M}$ (depending on the sample path $\omega$) such that 
\beq\label{eqn:limM0_norm_Phi}
\lim_{m\in\mc{M}}\norm{\nabla{\Phi(x^m,\rho^m,\bar{\mu}^m)}}=0.
\eeq
Since 
\beq\label{eqn:grad_Phi_expend}
\ba
\nabla{\Phi(x^m,\rho^m,\bar{\mu}^m)}&=\nabla{f(x^m)} + \sum_{i\in\I_0}\max\bigcurl{0,\bar{\mu}^m_i+\rho^mg_i(x^m)}\nabla{g_i(x^m)} \\
&=\nabla{f(x^m)} + \sum_{i\in\I_0}\mu^{m+1}_i\nabla{g_i(x^m)},
\ea
\eeq
\eqref{eqn:limM0_norm_Phi} can be equivalently written as 
\beq\label{eqn:limM0_norm_Phi_expend}
\lim_{m\in\mc{M}}\norm{\nabla{f(x^m)} + \sum_{i\in\I_0}\mu^{m+1}_i\nabla{g_i(x^m)} }=0.
\eeq
We split the discussion for two cases. \newline
Case 1. Suppose $\{\rho^m:m\in\mc{M}\}$ are bounded. Then by the condition in \Cref{lin:rho-update}
of \Cref{alg:penalty} for updating $\rho^m$, for large enough $m\in\mc{M}$ we should have
\bdm
\sigma_{N^m}(x^m,\mu^{m+1})\le\eta^m,
\edm
and hence
\bdm
\lim_{m\in\mc{M}}\sigma_{N^m}(x^m,\mu^{m+1})=0.
\edm
It implies the following limits hold:
\beq\label{eqn:lim_g_mu_1}
\ba
&\lim_{m\in\mc{M}}\min\bigcurl{-g_{0,N^m}(x^m),\mu^{m+1}_0}=0, \\
&\lim_{m\in\mc{M}}\min\bigcurl{-g_{i}(x^m),\mu^{m+1}_i}=0\quad\forall i\in\I.
\ea
\eeq
Since we also have 
\beq\label{eqn:lim_g}
\ba
&\lim_{m\to\infty}g_i(x^m)=g_i(x^*)\quad i\in\I, \\
&\lim_{m\to\infty}g_{0,N^m}(x^m)=g_0(x^*)\quad a.s.\textrm{ on $E^*$ by \Cref{lem:N-for-gN-convg}}.
\ea
\eeq 
Combining \eqref{eqn:lim_g_mu_1} and \eqref{eqn:lim_g} gives
\beq\label{eqn:lim_M_-g_mu}
\lim_{m\in\mc{M}}\min\bigcurl{-g_{i}(x^m),\mu^{m+1}_i}=\lim_{m\in\mc{M}}\min\bigcurl{-g_{i}(x^*),\mu^{m+1}_i}=0\quad\forall i\in\I_0.
\eeq
Since $g_i(x^*)<0$ for $i\in{J}$, the above limit implies that
\beq\label{eqn:lim_mu_J}
\lim_{m\in\mc{M}}\mu^{m+1}_i=0\quad a.s.\textrm{ on $E^*$}\;\forall i\in J.
\eeq
Substituting \eqref{eqn:lim_mu_J} into \eqref{eqn:limM0_norm_Phi_expend} shows that
\eqref{eqn:lim_I0-J} holds in Case 1. \newline
Case 2. $\{\rho^m:m\in\mc{M}\}$ is unbounded, i.e., $\rho^m\to\infty$.
Since $g_i(x^*)<0$ for $i\in{J}$ and $0\le\bar{\mu}^m_i\le\mumax$,
using \eqref{eqn:lim_g} shows that for all $m\in\mc{M}$ large enough we have
\bdm
\ba
&\bar{\mu}^m_i+\rho^mg_i(x^m)<0\quad i\in\I\cap{J},\\
&\bar{\mu}^m_0+\rho^mg_{0,N^m}(x^m)<0\quad a.s.\textrm{ on $E^*$}\quad\textrm{ if } 0\in{J}.
\ea
\edm
Therefore, $\lim_{m\in\mc{M}}\mu^{m+1}_i=0$ for $i\in{J}$, which shows
that \eqref{eqn:lim_I0-J} holds in Case 2. This concludes the proof of Part (1).

Part (2). Assume the sequence $\{\mu^{m+1}:m\in\mc{M}\}$ is not bounded.
Let $B_m=\norm{\mu^{m+1}}_\infty$. 
Then the sequence $\{\mu^{m+1}/B_m:m\in\mc{M}\}$
is bounded and hence there exists a subsequence $\mc{M}_1\subseteq\mc{M}$
such that 
\beq\label{eqn:lim_B_mu/B}
\lim_{m\in\mc{M}_1}B_m=\infty,\quad \lim_{m\in\mc{M}_1}\mu^{m+1}/B_m=\mu^\prime\textrm{ for some }\mu^\prime\ge{0}.
\eeq
Combining \eqref{eqn:lim_I0-J} and \eqref{eqn:lim_B_mu/B} gives
\beq\label{eqn:limit_norm=0}
\ba
0&=\lim_{m\in\mc{M}_1}\frac{1}{B_m}\norm{\nabla{f(x^m)} + \sum_{i\in\I_0\setminus{J}}\mu^{m+1}_i\nabla{g_i(x^m)} } \\
&=\norm{\sum_{i\in\I_0\setminus{J}}\mu^\prime_i\nabla{g_i(x^*)}}
\ea
\eeq
Since the MFCQ holds at $x^*$, there exists a vector $d$ such that $\inner{\nabla{g_i(x^*)}}{d}<0$
for all $i\in\I_0\setminus{J}$ (\Cref{def:constr-qualify}). Therefore,
\bdm
\ba
\inner{\sum_{i\in\I_0\setminus{J}}\mu^\prime_i\nabla{g_i(x^*)}}{d}
=\sum_{i\in\I_0\setminus{J}}\mu^\prime_i\inner{\nabla{g_i(x^*)}}{d}<0,
\ea
\edm
which contradicts \eqref{eqn:limit_norm=0}. It shows that $\{\mu^{m+1}: m\in\mc{M}\}$ are bounded.
Since $\{\mu^{m+1}: m\in\mc{M}\}$ are bounded and $x^*$ is feasible, any 
limit point $\mu^\prime$ of $\{\mu^{m+1}: m\in\mc{M}\}$ is a set of KKT multipliers 
associated with $x^*$, where the complementary slackness condition is ensured by
\eqref{eqn:lim_M_-g_mu}. If $\mu^\prime=\mu^*$ is unique, it follows that
$\{\mu^{m+1}: m\in\mc{M}\}$ must be a convergent sequence 
and $\lim_{m\in\mc{M}}\mu^{m+1}=\mu^*$. This concludes the proof of Part (2).
\qed
\end{proof}

The following theorem gives the properties of a limit point $x^*$ of the probabilistic sequence $\{x^m\}$
generated by \Cref{alg:penalty} in different cases.

\begin{theorem}\label{thm:convg}
Suppose \Cref{ass:f-c2-diff,ass:c1-diff,ass:Q-diff,ass:F_twice_diff} hold and 
\Cref{alg:penalty} has been applied to solve \eqref{opt:NLP}.
Let $\{x^{m},\rho^{m},\mu^{m},\bar{\mu}^{m}\}$ be the sequence 
generated from \Cref{alg:penalty}. Let $x^*$ be a point and
let $E^*$ be the event
\beq
E^*=\{\omega\in\Omega: \{x^m(\omega)\}\textrm{ has a limit point }x^*\}.
\eeq
Suppose the sample size in \Cref{alg:trust-region-merit-func} 
satisfies \eqref{eqn:N_cond_kappa_alpha} in each iteration,
the trust-region radius $r^m$ in Step 1 of  \Cref{alg:penalty} satisfies 
the condition in \Cref{lem:|gradPhi|-to-0} and the number of samples $N^m$ 
in Step 2 of \Cref{alg:penalty} satisfies the condition in \Cref{lem:N-for-gN-convg} for all iterations.
Then the following statements hold almost surely on $E^*$:
\newline
(1) If $\{\rho^m\}$ is bounded, $x^*$ is feasible.
\newline
(2) $x^*$ is a stationary point of the function $\norm{\max\{0,g(x)\}}^2$, 
where the $\max$ operator acts on every component of $g$.
\newline
(3) If $x^*$ is feasible and satisfies the CPLD constraint qualification of \eqref{opt:NLP},
then $x^*$ fulfills the KKT conditions of \eqref{opt:NLP}.
\end{theorem}

\begin{proof}
Part (1). Since $\{\rho^m\}$ is bounded, using the same argument as
in Case 1 of \Cref{lem:onE*converg} Part (1), we can obtain \eqref{eqn:lim_M_-g_mu},
which states that almost surely on $E^*$ there exists a subsequence indexed by $\mc{M}$ such that
\beq\label{eqn:comp_slack}
\lim_{m\in\mc{M}}\min\bigcurl{-g_{i}(x^*),\mu^{m+1}_i}=0\quad\forall i\in\I_0.
\eeq
Given that $\mu^{m+1}\ge{0}$ by the updating rule from \eqref{eqn:params-update}, the above limit implies that
\bdm
g_i(x^*)\le{0}\quad\forall{i\in\I_0},
\edm
which proves the feasibility of $x^*$.

Part (2). Let $F(x):=\norm{\max\{0,g(x)\}}^2=\sum_{i\in\I_0}|\max\{0,g_i(x)\}|^2$. The gradient of $F$
has the form
\beq
\nabla{F(x)}=\sum_{i\in\I_0}\max\{0,g_i(x)\}\nabla{g_i(x)}.
\eeq
If $x^*$ is feasible, we have $\max\{0,g_i(x^*)\}=0$ and hence $\nabla{F(x^*)}=0$.
Suppose $x^*$ is infeasible. By the result of Part (1), $\{\rho^m\}$ must be unbounded.
Using \eqref{eqn:grad_Phi_expend} and \eqref{eqn:grad_Phi_expend}, we conclude
that almost surely on $E^*$ there exist a subsequence $\mc{M}$ satisfying
\beq\label{eqn:grad_Phi_expend2}
\ba
0&=\lim_{m\in\mc{M}}\norm{\nabla{\Phi(x^m,\rho^m,\bar{\mu}^m)}} \\
&=\lim_{m\in\mc{M}}\norm{\nabla{f(x^m)} + \sum_{i\in\I_0}\max\bigcurl{0,\bar{\mu}^m_i+\rho^mg_i(x^m)}\nabla{g_i(x^m)}}.
\ea
\eeq
Divide both sides of \eqref{eqn:grad_Phi_expend2} by $\rho^m$ yields
\bdm
\lim_{m\in\mc{M}}\norm{\frac{1}{\rho^m}\nabla{f(x^m)} + \sum_{i\in\I_0}\max\bigcurl{0,\frac{\bar{\mu}^m_i}{\rho^m}+g_i(x^m)}\nabla{g_i(x^m)}}=0.
\edm
Since $x^m\to{x^*}$, $\bar{\mu}^m_i\le\mumax$ and $\{\rho^m\}$ is unbounded, the above limit gives
\bdm
\norm{\sum_{i\in\I_0}\max\{0,g_i(x^*)\}\nabla{g_i(x^*)}}=0.
\edm

Part (3). Let $\mc{M}$ be the index set of a subsequence satisfying $\lim_{m\in\mc{M}}x^m=x^*$
and \eqref{eqn:limM0_norm_Phi_expend}. 
We have that almost surely on $E^*$, 
\beq\label{eqn:lim_I0-J_recite}
\lim_{m\in\mc{M}}\norm{\nabla{f(x^m)} + \sum_{i\in\I_0}\mu^{m+1}_i\nabla{g_i(x^m)} }=0.
\eeq
Let 
\beq
v^m=\nabla{f(x^m)} + \sum_{i\in\I_0}\mu^{m+1}_i\nabla{g_i(x^m)}.
\eeq
Case 1. The subsequence $\left\{\norm{\mu^{m+1}}_\infty:m\in\mc{M}\right\}$ is bounded. 
Then there exists a convergent subsequence indexed by some $\mc{M}_1$ such that
$\lim_{m\in\mc{M}_1}\mu^{m+1}=\mu^\prime$ for some $\mu^\prime\ge{0}$. 
The limit \eqref{eqn:lim_I0-J_recite} implies that 
\bdm
\nabla{f(x^*)} + \sum_{i\in\I_0}\mu^\prime_i\nabla{g_i(x^*)}=0.
\edm
The above limit (for stationary condition) and \eqref{eqn:comp_slack} (for complementary slackness) point 
out that $x^*$ fulfills the KKT condition in Case 1. 
\newline
Case 2. The subsequence $\left\{\norm{\mu^{m+1}}_\infty:m\in\mc{M}\right\}$ is unbounded.
Define 
\bdm
G^m=\sum_{i\in\I_0}\mu^{m+1}_i\nabla{g_i(x^m)}.
\edm
Since $\norm{G^m}\le\norm{\nabla{f(x^m)}}+\norm{v^m}$, $\{G^m\}$ is bounded.
By Caratheodory's theorem, there exist a subset $\I^m\subset\I_0$ 
and a vector $\hat{\mu}^{m+1}\in\mR^{|\I^m|}_+$
such that $\{\nabla{g_i(x^m)}:i\in\I^m\}$ are linearly independent and
\beq\label{eqn:Gn=f+g}
G^m=\sum_{i\in\I^m}\hat{\mu}^{m+1}_i\nabla{g_i(x^m)}\quad\forall{m}\in\mc{M}.
\eeq
Taking an appropriate subsequence $\mc{M}_1$, we may assume that $\I^\prime=\I^m$ 
for large enough $m\in\mc{M}_1$, as because there are only finitely many subsets of $\I_0$. 
Let $B^m=\max\left\{1,\norm{\hat{\mu}^{m+1}}_\infty \right\}$. 
Divide \eqref{eqn:Gn=f+g} by $B_m$ gives that 
\beq\label{eqn:Gn/Bn}
\frac{G^m}{B^m}=\sum_{i\in\I^\prime}\frac{\hat{\mu}^{m+1}_i}{B^m}\nabla{g_i(x^m)}\quad\forall{m}\in\mc{M}.
\eeq
Since $\{\hat{\mu}^{m+1}/B^m:m\in\mc{M}\}$ is bounded, it has a convergent subsequence $\mc{M}_1$ such that
$\hat{\mu}^{m+1}/B^m\to\mu^\prime$ for some $\mu^\prime=1\ge{0}$ but $\norm{\mu^\prime}_\infty>0$.  
Taking the limit for $m\in\mc{M_1}$ on both sides of \eqref{eqn:Gn/Bn} and using 
$\norm{G^m/B^m}\to{0}$ leads to
\beq
\sum_{i\in\I^\prime}\mu^\prime_i\nabla{g_i(x^*)}=0.
\eeq
It shows that $\{\nabla{g_i(x^*)}:i\in\I^\prime\}$ are linearly dependent, and hence
due to the CPLD condition, $\{\nabla{g_i(x)}:i\in\I^\prime\}$ are linearly dependent for any $x$
in a neighborhood $U$ of $x^*$. But this contradicts \eqref{eqn:Gn=f+g}.
\qed
\end{proof}

\section{Preliminary Numerical Investigation}\label{sec:num}
\Cref{alg:trust-region-merit-func,alg:penalty} proposed in this work 
have been implemented in MATLAB and are available online~\cite{zenodo}.
We have tested the numerical performance of our methods on instances of
three benchmark CCPs: a nonconvex quantile optimization problem,
a portfolio optimization problem, and a joint chance-constrained optimization
problem.
\begin{example}[nonconvex1D \cite{wachter2020-chance-constr-prob-smooth-sample-nonlinear-approx}]
\label{exp:nonconvex_1d}
  The quantile optimization problem can be reformulated as a CCP:
\beq
\minimize_{x,y}\; y\;\; \text{s.t.: }\;\mb{P}[c(x,\xi)\le y]\ge 1-\alpha,
\eeq
where $c(x,\xi) \defined 0.25x^4-1/3x^3-x^2+0.2x-19.5+\xi_1x+\xi_2$ is a non-convex univariate 
function, with $\xi_1\sim N(0,3)$ and $\xi_2\sim N(0,144)$ as independent
random parameters. This problem is equivalent to minimizing the ($1-\alpha$)-quantile
of $c(x,\xi)$ over $x$.
\end{example}

\begin{example}[portfolio]\label{exp:portfolio}
Consider the portfolio optimization problem instance with a single individual linear chance constraint:
\beq
\maximize\; t \;\;\text{s.t.: } \mb{P}\{\xi^\top x\ge t\}\ge 1-\alpha,\; \sum^n_{i=1}x_i= 1,\; x_i\ge0,
\eeq
where $x_i$ denotes the fraction of investment in stock $i\in\{1,\ldots,n\}$,
$\xi_i\sim\mc{N}(\mu_i,\sigma^2_i)$
is a normally distributed random variable of return
with $\mu_i \defined 1.05+0.3\frac{n-i}{n-1}$ 
and $\sigma_i \defined \frac{1}{3}\left(0.05+0.6\frac{n-i}{n-1}\right)$.
Because the quantity $\xi^\top x$ is also a Gaussian random
variable with the mean and standard deviation being analytical functions
of $x$, the above problem can be reformulated as the second-order-cone
program for the case $\alpha<0.5$:
\begin{equation}\label{eqn:portfolio-convex-reform}
\begin{aligned}
&\maximize\;\; \sum^n_{i=1}\mu_ix_i + q^\alpha \sqrt{\sum^n_{i=1}\sigma^2_ix^2_i} 
\;\;\textrm{s.t.: }\; \sum^n_{i=1}x_i = 1,\; x_i\ge 0 \; \forall i\in[n],
\end{aligned}
\end{equation}
where $q^\alpha$ is the $\alpha$-quantile of the standard Gaussian distribution.
\end{example}

\begin{example}[jointChance \cite{luedtke2021-stoc-approx-frontier-chance-constr-nonlinear-program}]
\label{exp:joint_chance}
  This is an example of the $\ell_1$-norm 
  optimization with a joint chance constraint
\beq
\max \;\sum^n_{i=1}x_i\;\; \text{s.t. }\mb{P}\left\{\sum^n_{i=1}\xi^2_{ij}x^2_i\le U,\; j=1,\ldots,m \right\}\ge 1-\alpha,\; x_i\ge 0\;\forall i\in[n],
\eeq
where $\xi_{ij}$s are dependent normal random variables with mean $j/m$ and variance 1, and 
$\text{cov}(\xi_{ij},\xi_{i^\prime j})=0.5$, $\text{cov}(\xi_{ij},\xi_{i^\prime j^\prime})=0$ if $j\neq j^\prime$.
We set $m=5$ in our numerical experiments. 
\end{example}
Note that the above three examples have also been used in  
\cite{wachter2020-chance-constr-prob-smooth-sample-nonlinear-approx,
luedtke2021-stoc-approx-frontier-chance-constr-nonlinear-program} 
to test the performance of their methods. 
We adopted the following settings in the implementation of \Cref{alg:trust-region-merit-func,alg:penalty} for practical applications:
\begin{enumerate}
  	\item Since the sample size \eqref{eqn:N-cond-as-converg} and \eqref{eqn:N_cond_kappa_alpha} 
	 required for theoretical convergence can be prohibited,
	 we only let the sample size grow to a pre-specified maximum level ($N=10,000$ for example) 
	 and do not include additional samples in further iterations in the implementation.  
	\item The theoretical convergence of our algorithms is independent of the choice of Hessian matrix approximation. 
	In the implementation we adopt a neighborhood sampling method used by the POUNDERS derivative-free optimization method~\cite{SWCHAP14}
	to build a local quadratic model and extract the Hessian of the quadratic model as an estimation.
	\item The $\beta$ value (step size) for the finite-difference estimation of
    the empirical quantile function is set to a constant throughout the algorithm.  
    Theoretical convergence of the algorithm requires $\beta\to{0}$, however, it has to be supported
    by a sufficiently large sample size $N$. Since it is impractical to let $N\to\infty$ in the implementation,
    $\beta$ cannot be too small. In \Cref{tab:stepsize_test}, we compare the results from using different values for $\beta$.
\end{enumerate}
Other parameters in the two algorithms are
$\theta_\mu=0.5$ (the quadratic-penalty-reduction parameter), 
$\inrate=2.0$, $\derate=0.5$ (trust-region-size-control parameters),
and $\eta_1=0.1$, $\eta_2=0.25$ (trial move acceptance parameters). 
The finite-difference parameters $\beta=1.0\times 10^{-3}$ for experiments reported in \Cref{tab:comp_perf}.
Note that since the two methods are very different, comparison of the computational time
is not very sensible as they have very different setup in parameters which matters in determining
the overall performance. Therefore the computational time provided in \Cref{tab:comp_perf} is mainly for reference. 
The termination value of the trust-region radius $r^n$ and constraint-violation tolerance 
$\eta^n$ in \Cref{alg:penalty} are fixed to $10^{-5}$ for each iteration.
\begin{table}
\setlength{\tabcolsep}{3pt}
\centering
\captionsetup{font=scriptsize}
\caption{Computational performance of \Cref{alg:penalty} (labeled `finite difference') 
and quantile smoothing method 
\cite{wachter2020-chance-constr-prob-smooth-sample-nonlinear-approx} (labeled `smoothing')
on example problem instances using sample sizes $N=5,000, 10,000,$ and $20,000$.
For both methods, the final objective values reported in the table are evaluated
at the returned solution based on the same set of 50,000 samples.
\label{tab:comp_perf} }
\small
 \resizebox{\columnwidth}{!}{
\begin{tabular}{rrr|rr|rr|rr|rr|rr|rr}
\hline\hline
\multicolumn{3}{c|}{} & \multicolumn{4}{c|}{$N=5,000$}  &
\multicolumn{4}{c|}{$N=10,000$}  &  \multicolumn{4}{c}{$N=20,000$} \\
\hline
\multicolumn{3}{c|}{} & \multicolumn{2}{c|}{finite difference} & \multicolumn{2}{c|}{smoothing} & \multicolumn{2}{c|}{finite difference} & \multicolumn{2}{c|}{smoothing} & \multicolumn{2}{c|}{finite difference} & \multicolumn{2}{c}{smoothing} \\
\hline
Ex. &	dim	&	$\alpha$	&	time	&	obj	&	time	&	obj	&	time	&	obj	&	time	&	obj	&	time	&	obj	&	time	&  obj	\\
\hline
\ref{exp:nonconvex_1d}	&	1	&	0.05	&	0.1	&	0.0075	&	0.1	&	0.0078	&	0.1	&	-1.3069	&	0.1	&	-0.1396	&	0.1	&	-1.2983	&	0.1	&	0.2382	\\
\ref{exp:nonconvex_1d}	&	1	&	0.1	&	0.1	&	-4.5711	&	0.1	&	-4.5704	&	0.1	&	-4.5788	&	0.1	&	-4.5632	&	0.1	&	-4.5787	&	0.1	&	-4.1222	\\
\ref{exp:nonconvex_1d}	&	1	&	0.15	&	0.1	&	-8.8185	&	0.1	&	-8.8184	&	0.1	&	-7.5501	&	0.1	&	-7.0628	&	0.1	&	-8.8633	&	0.1	&	-7.0628	\\
\ref{exp:portfolio}	&	50	&	0.05	&	0.5	&	1.2266	&	0.1	&	1.2259	&	0.8	&	1.2271	&	0.1	&	1.2240	&	1.1	&	1.2278	&	0.2	&	1.2243	\\
\ref{exp:portfolio}	&	50	&	0.1	&	0.5	&	1.2445	&	0.1	&	1.2436	&	0.6	&	1.2451	&	0.1	&	1.2444	&	1.2	&	1.2455	&	0.1	&	1.2452	\\
\ref{exp:portfolio}	&	50	&	0.15	&	0.4	&	1.2568	&	0.2	&	1.2558	&	0.7	&	1.2576	&	0.1	&	1.2574	&	1.2	&	1.2587	&	0.2	&	1.2571	\\
\ref{exp:portfolio}	&	100	&	0.05	&	0.9	&	1.2502	&	0.3	&	1.2497	&	1.5	&	1.2513	&	0.4	&	1.2497	&	2.6	&	1.2508	&	0.3	&	1.2503	\\
\ref{exp:portfolio}	&	100	&	0.1	&	0.8	&	1.2630	&	0.3	&	1.2605	&	1.3	&	1.2645	&	0.3	&	1.2606	&	2.3	&	1.2639	&	0.3	&	1.2607	\\
\ref{exp:portfolio}	&	100	&	0.15	&	0.9	&	1.2741	&	0.3	&	1.2728	&	1.5	&	1.2754	&	0.3	&	1.2724	&	2.3	&	1.2759	&	0.3	&	1.2757	\\
\ref{exp:portfolio}	&	150	&	0.05	&	1.9	&	1.2615	&	0.6	&	1.2603	&	2.6	&	1.2623	&	0.5	&	1.2601	&	4.2	&	1.2621	&	0.6	&	1.2611	\\
\ref{exp:portfolio}	&	150	&	0.1	&	1.6	&	1.2738	&	0.5	&	1.2738	&	2.7	&	1.2751	&	0.5	&	1.2739	&	3.8	&	1.2747	&	0.5	&	1.2744	\\
\ref{exp:portfolio}	&	150	&	0.15	&	1.7	&	1.2831	&	0.5	&	1.2816	&	2.6	&	1.2844	&	0.5	&	1.2838	&	5.2	&	1.2840	&	0.5	&	1.2832	\\
\ref{exp:portfolio}	&	200	&	0.05	&	3.1	&	1.2687	&	1.0	&	1.2681	&	4.3	&	1.2697	&	0.8	&	1.2694	&	6.2	&	1.2700	&	0.8	&	1.2698	\\
\ref{exp:portfolio}	&	200	&	0.1	&	2.0	&	1.2811	&	1.2	&	1.2812	&	3.4	&	1.2814	&	0.8	&	1.2798	&	5.8	&	1.2816	&	0.7	&	1.2810	\\
\ref{exp:portfolio}	&	200	&	0.15	&	2.1	&	1.2900	&	0.9	&	1.2887	&	4.0	&	1.2896	&	0.7	&	1.2896	&	6.0	&	1.2894	&	0.7	&	1.2897	\\
\ref{exp:joint_chance}	&	10	&	0.05	&	4.6	&	10.6254	&	5.4	&	10.6518	&	12.4	&	10.6457	&	11.2	&	10.6443	&	13.6	&	10.6109	&	22.2	&	10.6155	\\
\ref{exp:joint_chance}	&	10	&	0.1	&	2.5	&	9.7844	&	5.2	&	9.8126	&	6.9	&	9.7675	&	14.3	&	9.7819	&	14.4	&	9.7727	&	29.3	&	9.7863	\\
\ref{exp:joint_chance}	&	10	&	0.15	&	4.0	&	9.2730	&	5.4	&	9.2873	&	15.9	&	9.2666	&	10.2	&	9.2747	&	16.6	&	9.2644	&	17.5	&	9.2763	\\
\ref{exp:joint_chance}	&	20	&	0.05	&	5.3	&	22.1048	&	8.9	&	22.1455	&	11.3	&	22.0936	&	16.5	&	22.1071	&	37.3	&	22.0292	&	24.5	&	22.0510	\\
\ref{exp:joint_chance}	&	20	&	0.1	&	2.7	&	20.5364	&	7.5	&	20.5118	&	6.2	&	20.5726	&	15.2	&	20.5503	&	20.7	&	20.5280	&	23.4	&	20.4961	\\
\ref{exp:joint_chance}	&	20	&	0.15	&	3.5	&	19.5481	&	8.2	&	19.5869	&	5.0	&	19.5241	&	18.2	&	19.5604	&	12.0	&	19.4399	&	17.6	&	19.4709	\\
\ref{exp:joint_chance}	&	30	&	0.05	&	9.6	&	33.1786	&	9.7	&	33.4028	&	12.8	&	33.2837	&	24.2	&	33.4344	&	16.0	&	33.3910	&	22.0	&	33.4280	\\
\ref{exp:joint_chance}	&	30	&	0.1	&	4.0	&	30.9151	&	7.7	&	31.0496	&	6.2	&	30.9307	&	15.3	&	31.0203	&	15.4	&	31.0541	&	16.8	&	31.0904	\\
\ref{exp:joint_chance}	&	30	&	0.15	&	4.4	&	29.4587	&	8.2	&	29.5399	&	5.6	&	29.4262	&	8.6	&	29.4464	&	14.4	&	29.4763	&	17.0	&	29.4982	\\
\ref{exp:joint_chance}	&	40	&	0.05	&	6.2	&	44.3838	&	15.2	&	44.6468	&	19.6	&	44.5109	&	19.2	&	44.6539	&	27.9	&	44.4231	&	24.9	&	44.5204	\\
\ref{exp:joint_chance}	&	40	&	0.1	&	5.9	&	41.8184	&	10.7	&	41.8876	&	12.9	&	41.7944	&	15.1	&	41.7689	&	17.3	&	41.5963	&	32.4	&	41.5188	\\
\ref{exp:joint_chance}	&	40	&	0.15	&	6.1	&	39.7211	&	9.8	&	39.8218	&	21.0	&	39.6720	&	13.6	&	39.7234	&	15.1	&	39.5076	&	30.0	&	39.5701	\\
\hline
\end{tabular}}

\end{table}

All problem instances are generated from the three examples above by specifying
the problem dimension (for `portfolio' and `jointChance') and the risk value $\alpha$;
each instance is specified by the first three columns in \Cref{tab:comp_perf}.  
Two sets of numerical experiments were performed. 
First, we study the computational performance of the augmented Lagrangian method
with the empirical quantile value estimation and finite-difference estimation of the quantile gradient (ALM-quant)
on the problem instances with different sample sizes ($N=5,000,\; 10,000
\text{ and } 20,000$). We compare the performance of our method with 
a recasted version of the method from 
\cite{wachter2020-chance-constr-prob-smooth-sample-nonlinear-approx}, 
which is the most relevant one in terms of reformulating chance constraints. 
In \cite{wachter2020-chance-constr-prob-smooth-sample-nonlinear-approx},
a smoothing method is developed to estimate the gradient of the quantile function $\nabla\Qx$. 
In particular, it estimates $\nabla\Qx$ as 
\beq\label{eqn:grad_q}
\nabla{q(x)}=\frac{\sum^N_{i=1}\Gamma^\prime_\epsilon(c(x,\xi_i)-Q_{\epsilon,N}(x))\nabla{c(x,\xi_i)}}
{\sum^N_{i=1}\Gamma^\prime_\epsilon(c(x,\xi_i)-Q_{\epsilon,N}(x))},
\eeq
where $\Gamma_\epsilon$ is a smoothing function defined as
\bdm
\Gamma_\epsilon(y)=
	\begin{cases}
		1, \qquad y\le-\epsilon \\
		\gamma_\epsilon(y), \quad -\epsilon<y<\epsilon \\
		0,  \qquad y\ge\epsilon \\
	\end{cases}
\edm
and 
\bdm
\gamma_\epsilon(y)=\frac{15}{16}\left(-\frac{1}{5}\left(\frac{y}{\epsilon}\right)^5 + \frac{2}{3}\left(\frac{y}{\epsilon}\right)^3
	-\left(\frac{y}{\epsilon}\right) + \frac{8}{15} \right),
\edm
and $\Gamma^\prime_\epsilon$ is the derivative of $\Gamma_\epsilon$.
The $Q_{\epsilon,N}(x)$ in \eqref{eqn:grad_q} is a smoothed quantile value. 
To obtain $Q_{\epsilon,N}(x)$,
one needs to numerically solve a nonlinear equation involving $\Gamma_\epsilon$.  
This is not convenient to implement, and therefore we replace $Q_{\epsilon,N}(x)$ with
the empirical quantile $\wh{Q}^{1-\alpha}_N(x)$ in the computation of $\nabla{q}(x)$.
Note that \eqref{eqn:grad_q} for computing the gradient of the smoothed quantile function
can be used in many different optimization algorithms.
In \cite{wachter2020-chance-constr-prob-smooth-sample-nonlinear-approx}, 
it is used in an interior-point method to carry out the optimization. 
To compare it with \Cref{alg:penalty}, we use it in the augmented Lagrangian method 
(the same base method as in \Cref{alg:penalty}), and hence only the way 
of computing the quantile gradient is different between \Cref{alg:penalty} and 
the smoothing method. The candidate values of $\epsilon$ is 
the same as $\beta$ in our comparison.
The computational results of our method and the smoothing method are 
in \Cref{tab:comp_perf}.

For each method of estimating the gradient of the quantile function (finite difference vs. smoothing), 
it turns out that the final objective values yielded by the method
are not very sensitive to the increase of sample size from 5000 to 20,000.
However, a larger sample size often leads to longer computational time, 
which is expected because the quantile estimation and gradient calculation can take more time.
Overall, the two methods are competitive in terms of the computational time and the quality of the solution. 
The finite-difference method gives better results for the nonconvex1D problem instances.
For the portfolio problem instances, the objective values obtained by the two methods are very close,
but the results generated by the finite-difference method are slightly better.
The smoothing method gives better results for the joint-chance problem in 30 out of 36 instances.
It means that the smoothing method has some advantage in joint-chance constrained problems
with higher-dimensional random parameters.

Since the portfolio problem can be reformulated as a convex optimization problem \eqref{eqn:portfolio-convex-reform},
its global optimal objective is obtainable, and we can further study the optimality
gap of the best objective identified by the ALM-quant. 
The results are summarized in \Cref{tab:opt-gap}. 
The optimality gap is in the range of $0.06\%\sim 0.18\%$,
indicating that the ALM-quant method leads to high-quality solutions for this specific category of problem instances.

\begin{table}
\centering
\captionsetup{font=scriptsize}
\caption{Comparison between the optimal objective and the objective value identified by
the ALM-quant method with the setting the same as for the \Cref{tab:comp_perf}
and $N=10,000$. The final objective values reported in the table are evaluated
at the returned solution based on the same set of 50,000 samples.}
\label{tab:opt-gap}
\scriptsize
\begin{tabular}{rrr|rrr}
\hline\hline
Ex. &	dim	&	risk	&	opt obj	&	best obj	&	gap(\%)	\\
\hline
\ref{exp:portfolio}	&	50	&	0.05	&	1.2291	&	1.2271	&	0.16272	\\
\ref{exp:portfolio}	&	50	&	0.1	&	1.2468	&	1.2451	&	0.13595	\\
\ref{exp:portfolio}	&	50	&	0.15	&	1.2600	&	1.2576	&	0.18667	\\
\ref{exp:portfolio}	&	100	&	0.05	&	1.2521	&	1.2513	&	0.06341	\\
\ref{exp:portfolio}	&	100	&	0.1	&	1.2666	&	1.2645	&	0.16651	\\
\ref{exp:portfolio}	&	100	&	0.15	&	1.2773	&	1.2754	&	0.14570	\\
\ref{exp:portfolio}	&	150	&	0.05	&	1.2637	&	1.2623	&	0.10825	\\
\ref{exp:portfolio}	&	150	&	0.1	&	1.2765	&	1.2751	&	0.11148	\\
\ref{exp:portfolio}	&	150	&	0.15	&	1.2860	&	1.2844	&	0.12309	\\
\ref{exp:portfolio}	&	200	&	0.05	&	1.2711	&	1.2697	&	0.10794	\\
\ref{exp:portfolio}	&	200	&	0.1	&	1.2829	&	1.2814	&	0.11755	\\
\ref{exp:portfolio}	&	200	&	0.15	&	1.2915	&	1.2896	&	0.14704	\\
\hline
\end{tabular}

\end{table}

For the second batch of experiments, we focus on testing the impact of the finite-difference parameter $\beta$
on the quality of solution for the problem instances 
with $\alpha$ being 0.05, 0.1 and 0.15 and sample size $N$ being 5,000 and 10,000, respectively. 
For each problem instance, 
five different finite-difference parameters were considered
($\beta=1.0\times 10^{-4},\;5.0\times 10^{-4},\; 1.0\times 10^{-3}\; 5.0\times 10^{-3}\text{ and }1.0\times 10^{-2}$),
and the objective values identified by the ALM-quant under these $\beta$ values are reported in
\Cref{tab:stepsize_test}. All other parameters are set 
the same as the previous experiments. It can be observed that for most of the portfolio and joint-chance 
problem instances, the objective values yielded by
smaller $\beta$ values (i.e., $1\times10^{-4}$ and $5\times10^{-4}$) are less competitive as
that yielded by larger $\beta$ values (i.e., $1\times10^{-3}$, $5\times10^{-3}$ and $1\times10^{-2}$). 
To interpret this outcome, we realize that 
there are two sources of errors: the randomness error of quantile evaluation and the numerical
error of finite differencing for the gradient estimation. The randomness error by itself only depends 
on the sample size but it will be magnified according to the step size when it is propagated into
the quantile-gradient estimation. That is, the contribution of the randomness error in the quantile-gradient
estimation is roughly $\delta/\beta$, where $\delta$ is the randomness error of the quantile estimation.
Therefore, while reducing the finite-difference parameter can decrease the numerical error 
of gradient estimation in the finite difference calculation (in the case that the function value is error free),
the randomness error can be magnified significantly in the end. This implies that a relatively larger
finite-difference parameter can lead to a better quantile gradient estimation and hence a better solution.
From a different angle, having a larger finite-difference parameter can be interpreted as an implicit smoothing.

\begin{table}
\centering
  \setlength{\tabcolsep}{3pt}
\captionsetup{font=scriptsize}
\caption{\scriptsize Comparison of objective values identified by 
\Cref{alg:penalty} with five different $\beta$ parameters for computing the finite difference. 
We consider sample size $N=5000$ and $N=10,000$ for all the three numerical examples in this investigation.}\label{tab:stepsize_test}
\tiny
\resizebox{\columnwidth}{!}{
\begin{tabular}{rrr|ccccc|ccccc}
\hline\hline
\multicolumn{3}{c|}{} & \multicolumn{5}{c|}{$N=5000$}  & \multicolumn{5}{c}{$N=10000$}  \\
\hline
 Ex. & dim	& $\alpha$ & 1e-4 & 5e-4 & 1e-3 & 5e-3 & 1e-2 & 1e-4 & 5e-4 & 1e-3 & 5e-3 & 1e-2	\\
\hline
\ref{exp:nonconvex_1d}	&	1	&	0.05	&	0.0075	&	0.0075	&	0.0075	&	0.0078	&	0.0088	&	-1.2063	&	-1.2288	&	-1.3069	&	-0.1396	&	0.0834	\\
\ref{exp:nonconvex_1d}	&	1	&	0.1	&	-4.5710	&	-4.5710	&	-4.5709	&	-4.5710	&	-4.5711	&	-4.5788	&	-4.5788	&	-4.5788	&	-4.5787	&	-4.5788	\\
\ref{exp:nonconvex_1d}	&	1	&	0.15	&	-8.8184	&	-8.8184	&	-8.8185	&	-7.5508	&	-7.0628	&	-7.0628	&	-7.0628	&	-7.0628	&	-7.5500	&	-7.5501	\\
\ref{exp:portfolio}	&	50	&	0.05	&	1.2137	&	1.2134	&	1.2166	&	1.2255	&	1.2266	&	1.2008	&	1.2177	&	1.2221	&	1.2271	&	1.2265	\\
\ref{exp:portfolio}	&	50	&	0.1	&	1.2219	&	1.2317	&	1.2373	&	1.2445	&	1.2445	&	1.2238	&	1.2386	&	1.2386	&	1.2451	&	1.2440	\\
\ref{exp:portfolio}	&	50	&	0.15	&	1.2311	&	1.2457	&	1.2430	&	1.2556	&	1.2568	&	1.2416	&	1.2294	&	1.2534	&	1.2576	&	1.2566	\\
\ref{exp:portfolio}	&	100	&	0.05	&	1.2295	&	1.2220	&	1.2324	&	1.2475	&	1.2502	&	1.2281	&	1.2312	&	1.2365	&	1.2489	&	1.2513	\\
\ref{exp:portfolio}	&	100	&	0.1	&	1.2333	&	1.2482	&	1.2553	&	1.2630	&	1.2609	&	1.2441	&	1.2533	&	1.2542	&	1.2645	&	1.2645	\\
\ref{exp:portfolio}	&	100	&	0.15	&	1.2485	&	1.2542	&	1.2647	&	1.2741	&	1.2713	&	1.2582	&	1.2663	&	1.2653	&	1.2740	&	1.2754	\\
\ref{exp:portfolio}	&	150	&	0.05	&	1.2331	&	1.2295	&	1.2321	&	1.2577	&	1.2615	&	1.2428	&	1.2358	&	1.2544	&	1.2598	&	1.2623	\\
\ref{exp:portfolio}	&	150	&	0.1	&	1.2371	&	1.2525	&	1.2593	&	1.2738	&	1.2734	&	1.2444	&	1.2599	&	1.2678	&	1.2743	&	1.2751	\\
\ref{exp:portfolio}	&	150	&	0.15	&	1.2503	&	1.2641	&	1.2723	&	1.2824	&	1.2831	&	1.2571	&	1.2654	&	1.2752	&	1.2836	&	1.2844	\\
\ref{exp:portfolio}	&	200	&	0.05	&	1.2467	&	1.2381	&	1.2387	&	1.2629	&	1.2687	&	1.2460	&	1.2460	&	1.2491	&	1.2692	&	1.2697	\\
\ref{exp:portfolio}	&	200	&	0.1	&	1.2480	&	1.2538	&	1.2577	&	1.2772	&	1.2811	&	1.2439	&	1.2624	&	1.2721	&	1.2809	&	1.2814	\\
\ref{exp:portfolio}	&	200	&	0.15	&	1.2597	&	1.2762	&	1.2763	&	1.2859	&	1.2900	&	1.2646	&	1.2764	&	1.2757	&	1.2885	&	1.2896	\\
\ref{exp:joint_chance}	&	10	&	0.05	&	10.5462	&	10.5663	&	10.5982	&	10.6054	&	10.6254	&	10.6323	&	10.6136	&	10.6123	&	10.6326	&	10.6457	\\
\ref{exp:joint_chance}	&	10	&	0.1	&	9.6923	&	9.7922	&	9.7603	&	9.7844	&	9.7577	&	9.7804	&	9.7730	&	9.7575	&	9.7675	&	9.7530	\\
\ref{exp:joint_chance}	&	10	&	0.15	&	9.2522	&	9.2726	&	9.2730	&	9.2590	&	9.2505	&	9.2675	&	9.2651	&	9.2627	&	9.2640	&	9.2666	\\
\ref{exp:joint_chance}	&	20	&	0.05	&	22.1026	&	22.0852	&	22.0967	&	22.0697	&	22.1048	&	22.0360	&	21.9819	&	22.0714	&	22.0936	&	22.0892	\\
\ref{exp:joint_chance}	&	20	&	0.1	&	20.5542	&	20.5447	&	20.5364	&	20.3174	&	20.4485	&	20.5292	&	20.4929	&	20.5153	&	20.5726	&	20.5695	\\
\ref{exp:joint_chance}	&	20	&	0.15	&	19.5518	&	19.4862	&	19.5125	&	19.5481	&	19.5386	&	19.4471	&	19.5249	&	19.3867	&	19.5241	&	19.5229	\\
\ref{exp:joint_chance}	&	30	&	0.05	&	33.2087	&	33.0971	&	33.1786	&	33.1337	&	33.1509	&	33.2980	&	33.2524	&	33.2837	&	33.2788	&	33.2626	\\
\ref{exp:joint_chance}	&	30	&	0.1	&	30.9020	&	30.9052	&	30.9151	&	30.9065	&	30.8984	&	30.9187	&	30.9345	&	30.9072	&	30.9307	&	30.9109	\\
\ref{exp:joint_chance}	&	30	&	0.15	&	29.4373	&	29.4598	&	29.3957	&	29.4587	&	29.4254	&	29.3760	&	29.4152	&	29.4090	&	29.4262	&	29.3935	\\
\ref{exp:joint_chance}	&	40	&	0.05	&	44.3435	&	44.3007	&	44.2901	&	44.3741	&	44.3838	&	44.3854	&	44.4181	&	44.4693	&	44.5109	&	44.4467	\\
\ref{exp:joint_chance}	&	40	&	0.1	&	41.8191	&	41.8221	&	41.8184	&	41.8074	&	41.7176	&	41.7392	&	41.7637	&	41.7765	&	41.7944	&	41.7490	\\
\ref{exp:joint_chance}	&	40	&	0.15	&	39.7135	&	39.7125	&	39.7150	&	39.7123	&	39.7211	&	39.6771	&	39.6703	&	39.6643	&	39.6720	&	39.6550	\\
\hline
\end{tabular}}

\end{table}

\section{Concluding Remarks}
The finite-difference estimation of the quantile gradient has been incorporated
into an augmented Lagrangian method coupled with a trust-region algorithm to approach 
the nonlinear optimization problem with chance constraints. 
Convergence analysis has been established for this approach
and numerical results show that this approach is competitive with the explicitly smoothing method for
estimating the gradient of the quantile function for the problem instances considered. 
It is worth remarking that the augmented Lagrangian method serves as a carrier for
the estimation of quantile-function values and gradients. 
The estimation can certainly be used in other
algorithms for constrained optimization such as the interior-point method, and it can 
be directly used in NLP solvers. The performance of solving nonlinear chance-constrained 
problem instances in practice is a combination of the solver performance, 
the estimation accuracy of quantile values and quantile gradients, the sampling techniques,
smoothing techniques, and other ad-hoc strategies, which require additional empirical investigation.

\section*{Acknowledgment}
Careful comments from three anonymous referees led to substantial improvements of the main theoretical result and quality of the
numerical section.
This work was supported by the U.S.~Department of Energy, Office of Science,
Office of Advanced Scientific Computing Research, Scientific Discovery through
Advanced Computing (SciDAC) Program through the FASTMath Institute under
Contract No.~DE-AC02-06CH11357.

\bibliographystyle{siamplain}
\bibliography{references}

\framebox{\parbox{.90\linewidth}{\scriptsize The submitted manuscript has been created by
        UChicago Argonne, LLC, Operator of Argonne National Laboratory (``Argonne'').
        Argonne, a U.S.\ Department of Energy Office of Science laboratory, is operated
        under Contract No.\ DE-AC02-06CH11357.  The U.S.\ Government retains for itself,
        and others acting on its behalf, a paid-up nonexclusive, irrevocable worldwide
        license in said article to reproduce, prepare derivative works, distribute
        copies to the public, and perform publicly and display publicly, by or on
        behalf of the Government.  The Department of Energy will provide public access
        to these results of federally sponsored research in accordance with the DOE
        Public Access Plan \url{http://energy.gov/downloads/doe-public-access-plan}.}}
\end{document}